\documentclass[a4paper,11pt]{amsart}

\usepackage{amsfonts, latexsym, amsmath, amssymb, amsthm, amscd}
\usepackage[all]{xy}
\usepackage[english]{babel}
\usepackage[utf8]{inputenc}
\usepackage{graphicx}
\usepackage{booktabs}
\usepackage{array, tabularx}
\usepackage{textcomp}
\usepackage{tikz, tikz-cd}
\usepackage{multirow}
\usepackage{paralist}
\usepackage{comment}
\usepackage{url}
\usepackage{tensor}
\usepackage{MnSymbol}
\usepackage{color}
\usepackage{mathdots}
\usepackage[hypertexnames=false,
backref=page,
    pdftex,
    pdfpagemode=UseNone,
    breaklinks=true,
    extension=pdf,
    colorlinks=true,
    linkcolor=blue,
    citecolor=blue,
    urlcolor=blue,
]{hyperref}

\usepackage{enumitem}
\usepackage[top=1.5in, bottom=.9in, left=0.9in, right=0.9in]{geometry}

\theoremstyle{plain}
\newtheorem{theorem}{Theorem}[section]
\newtheorem*{theorem*}{Theorem}
\newtheorem{definition}[theorem]{Definition}
\newtheorem*{definition*}{Definition}
\newtheorem{lemma}[theorem]{Lemma}
\newtheorem{prop}[theorem]{Proposition}
\newtheorem*{prop*}{Proposition}
\newtheorem{cor}[theorem]{Corollary}
\newtheorem{rem}[theorem]{Remark}
\theoremstyle{definition}

\newtheorem*{mt*}{Main Theorem}

\DeclareMathOperator{\Span}{Span}
\DeclareMathOperator{\id}{id}
\DeclareMathOperator{\End}{End}

\DeclareMathSymbol{\Finv} {\mathord}{AMSb}{"60}

\newcommand\restrict[1]{\raisebox{-.5ex}{$|$}_{#1}}

\newcommand{\R}{\mathbb{R}}
\newcommand{\C}{\mathbb{C}}

\newcommand{\del}{\partial}

\numberwithin{equation}{section}

\let\phi\varphi

\DeclareFontFamily{U}{MnSymbolC}{}
\DeclareSymbolFont{MnSyC}{U}{MnSymbolC}{m}{n}
\DeclareFontShape{U}{MnSymbolC}{m}{n}{
    <-6>  MnSymbolC5
   <6-7>  MnSymbolC6
   <7-8>  MnSymbolC7
   <8-9>  MnSymbolC8
   <9-10> MnSymbolC9
  <10-12> MnSymbolC10
  <12->   MnSymbolC12}{}
\DeclareMathSymbol{\intprod}{\mathbin}{MnSyC}{'270}

\allowdisplaybreaks

\author{Richard Hind}
\address[Richard Hind]{Department of Mathematics \\
University of Notre Dame \\
Notre Dame, IN 46556}
\email{hind.1@nd.edu}

\author{Tommaso Sferruzza}
\address[Tommaso Sferruzza]{
Dipartimento di Scienze Matematiche, Fisiche e Informatiche\\
Unità di Mate\-matica e Informatica\\
Università degli studi di Parma}
\email{tommaso.sferruzza@unipr.it}

\author{Adriano Tomassini}
\address[Adriano Tomassini]{
Dipartimento di Scienze Matematiche, Fisiche e Informatiche\\
Unità di Mate\-matica e Informatica\\
Università degli studi di Parma}
\email{adriano.tomassini@unpir.it}

\title[Almost complex blow-ups and positive closed $(1,1)$-forms on $4$-dimensional \dots]{Almost complex blow-ups and positive closed $(1,1)$-forms on $4$-dimensional almost complex manifolds}

\thanks{The first author is partially supported by Simons Foundation grant \# 317510. The second author is partially supported by GNSAGA of INdAM. The third author is partially supported by the project PRIN2017 “Real and Complex Manifolds: Topology, Geometry and holomorphic dynamics” (code 2017JZ2SW5), and by GNSAGA of INdAM}

\date{\today}

\begin{document}
\maketitle

\begin{abstract}
Let $(M,J)$ be a $2n$-dimensional almost complex manifold and let $x\in M$. We define the notion of {\em almost complex blow-up} of $(M,J)$ at $x$. We prove the existence of almost complex blow-ups at $x$ under suitable assumptions on the almost complex structure $J$ and we provide explicit examples of such a construction. We note that almost complex blow-ups are unique. When $(M,J)$ is a $4$-dimensional almost complex manifold, we give an obstruction on $J$ to the existence of almost complex blow-ups at a point and prove that the almost complex blow-up at a point of a compact almost K\"ahler manifold is almost K\"ahler.
\end{abstract}
\tableofcontents
\section{Introduction}
Let $M$ be an $n$-dimensional complex manifold and $x$ be a point of $M$. The rough idea of the complex blow-up of $M$ at the point $x$ is to replace $x$ with a  copy of the complex $(n-1)$-dimensional projective space $\mathbb{P}^{n-1}(\C)$. Indeed, let 
$$
\mathcal{B}:=\Big\{\Big( (z_1,\ldots ,z_n),[v_1:\ldots ,v_n]\Big)\in\mathcal{U}\times\mathbb{P}^{n-1}(\C)\,\,\,\vert\,\,\, z_iv_j-z_jv_i=0\,\,\quad i,j=1,\ldots ,n\Big\}.
$$
then denoting by $p_1^{\mathcal{B}}$ the restriction to $\mathcal{B}$ of the projection $p_1$ on the first factor, we have that $E:=(p_1^{\mathcal{B}})^{-1}(0)=\mathbb{P}^{n-1}(\C)$. If $M$ is an $n$-dimensional complex manifold and $x$ is a point of $M$, then one can work locally taking a holomorphic chart $\phi:\mathcal{U}\to\phi(\mathcal{U})\subset M$, $0\in\mathcal{U}\subset\C^n$, $x=\phi(0)$; then $\phi\circ p_1^{\mathcal{B}}:\mathcal{B}\setminus E\to \phi(\mathcal{U})\setminus \{x\}$ and $\phi\circ p_1^{\mathcal{B}}$ is a diffeomorphism between a neighbourhood $V$ of 
$
\{0\}\times\mathbb{P}^{n-1}(\C)\subset\mathcal{B}
$
and a neighbourhood of $x$ in $M$. Then the {\em blow-up} of $M$ at $x$ is the complex manifold defined as 
$$
\tilde{M}_x=\left(M\setminus \{x\}\cup V\right)\slash \sim
$$
where the equivalence relation $\sim$ identifies $\left((z_1,\ldots,z_n),[v_1:,\ldots,:v_n]\right)\in V\setminus\mathbb{P}^{n-1}(\C)$ with $\phi((z_1,\ldots,z_n))$.\newline
Furthermore, for a given $k$-dimensional complex submanifold $Y$ of an $n$-dimensional complex manifold $M$, it is defined the complex blow-up $\tilde{M}_Y$ of $M$ along $Y$ 
(see e.g., \cite[p.602]{Gr-H}, \cite[p.199]{Bl}, \cite[p.78]{Voi}). It turns out that if $M$ is K\"ahler and $Y$ is a compact, then $\tilde{M}_Y$ is K\"ahler (see e.g., \cite[p.80]{Voi}).\newline 
Assume now that $(M,J)$ is a $2n$-dimensional almost complex manifold, that is $M$ is a $2n$-dimensional smooth manifold endowed with $J\in \End(TM)$ such that, $J^2=-id_{TM}$. The {\em Nijenhuis tensor} of $J$ is  the $(1,2)$-tensor field $N_J$ defined, for every pair of vector fields $X,Y$ on $M$, as 
$$
N_J(X,Y)=[JX,JY]-[X,Y],-J[JX,Y]-J[X,JY].
$$
Then in view of the celebrated Theorem by Newlander and Nirenberg, $J$ is integrable, that is it is induced by the structure of a complex manifold, if and only if the Nijenhuis tensor of $J$ vanishes identically. As proved by McDuff and Salamon (see \cite[p.302]{DS}), the blow-up construction at a point can be extended to the case of an almost complex manifold $(M,J)$, by assuming that the almost complex structure $J$ is integrable in a neighbourhood of $x$, or, equivalently, the Nijenhuis tensor $N_J$ of $J$ vanishes on a neighbourhood of $x$. In a similar way, the blow-up of a $2n$-dimensional almost complex manifold $(M,J)$ along a $J$-invariant $2k$-dimensional submanifold $Y$ of $M$ can be defined by assuming that $J$ is integrable in a neighbourhood of $Y$. For the notion of the {\em symplectic blow-up} we refer to the book by McDuff and Salamon \cite{DS} and the references therein.\vskip.1truecm

In the present paper we are interested in defining the blow-up at a point in the category of the almost complex manifolds, without assuming the integrability of the almost complex structure in a neighbourhood of the point $x$. 
\begin{definition}\label{blow-up-almost-complex-intro}
A {\em blow-up} of a $2n$-dimensional almost complex manifold $(M,J)$ at a point $x$ is a pair $((\tilde{M}_x,\tilde{J}),\pi)$, where $(\tilde{M}_x,\tilde{J})$ is a $2n$-dimensional almost complex manifold and $\pi:\tilde{M}_x\to M$ is a $(\tilde{J},J)$-holomorphic map such that $E=\pi^{-1}(x)$ is an embedded $\mathbb{P}^{n-1}(\C)$ and $\pi:\tilde{M}_x\setminus E\to M\setminus\{x\}$ is a bijection.  
\end{definition}

As reminded before, the existence of local holomorphic coordinates on a complex manifold $(M,J)$ allows to construct the blow-up at a point. For an arbitrary almost complex manifold $(M,J)$, with $J$ possibly non integrable, we have non local holomorphic coordinates. We start to fix a local smooth chart $(\mathcal{U},\phi)$ at $x$, for $\mathcal{U}\subset\C^n$, $\phi(0)=x$ and we consider the almost complex structure on $\C^n$, defined as $\mathbb{J}:=d\phi^{-1}\circ J\circ d\phi$. Then, as for the case when $J$ is integrable, we define $\mathcal{B}\subset\mathcal{U}\times\mathbb{P}^{n-1}(\C)$ and we take $p_1^{\mathcal{B}}:\mathcal{B}\to \mathcal{U}$ the restriction to $\mathcal{B}$ of the natural projection $p_1:\mathcal{U}\times\mathbb{P}^{n-1}(\C)\to \mathcal{U}$. It turns out that 
$p_1^{\mathcal{B}}:\mathcal{B}\setminus (p_1^{\mathcal{B}})^{-1}(0)\to \mathcal{U}\setminus \{0\}$ is a diffeomorphism and consequently 
$\tilde{\mathbb{J}}=d(p_1^{\mathcal{B}})^{-1}\circ \mathbb{J}\circ dp_1^{\mathcal{B}}
$ 
gives rise to an almost complex structure away from  $(p_1^{\mathcal{B}})^{-1}(0)=\mathbb{P}^{n-1}(\C)$. At this point, we need to extend the almost complex structure $\tilde{\mathbb{J}}$ across $(p_1^{\mathcal{B}})^{-1}(0)$. We will give sufficient conditions on the chart in order that $\tilde{J}$ has a smooth extension. We will see in Lemma \ref{topological-extension} that a blow-up exists if and only if $\tilde{J}$ has a smooth extension in some local chart (but not all see Remark \ref{non-sufficient}).

Usher in \cite[Section 6]{U} has shown that given any choice of chart, there exists a $\mathcal{C}^1$-atlas on $\mathcal{B}$ such that $\tilde{J}$ has a Lipschitz extension. \newline 
We introduce the following definition, see Definition \ref{strong-line-condition-2n}.
\begin{definition*}
Let $(M,J)$ be $2n$-dimensional almost complex manifold. The almost complex structure $J$ satisfies the {\em line condition at} $x\in M$ if there exists a smooth local chart $(\mathcal{U},\phi)$, $0\in\mathcal{U}\subset\C^n$, $\phi:\mathcal{U}\to \phi(\mathcal{U})\subset M$, $\phi(0)=x$, such that for all $(p,q)\in \C^n$ the almost complex structure coincides with $i$ on the complex subspace $Span_\C\langle p,q\rangle$ of $T_{(p,q)}\C^n$ spanned by $(p,q)$ and also coincides with $i$ on the quotient $T_{(p,q)}\C^n\slash Span_\C\langle p,q\rangle$.
\end{definition*}
We prove the following theorem, see Theorem \ref{theorem-blowup-strong-2n}.
\begin{theorem*}
Let $(M,J)$ be a $2n$-dimensional almost complex manifold. If $J$ satisfies the line condition at $x\in M$,  then there exists an almost complex blow-up $(\tilde{M}_x,\tilde{J})$ of $(M,J)$ at $x$.
\end{theorem*}
We also prove the following, see Theorem \ref{uniqueness-blowup-2n}. 
\begin{theorem*}
Almost complex blow-ups at a point are unique.
\end{theorem*}
Focusing on dimension $4$, we give a necessary condition for the existence of almost complex blow-ups, see Theorem \ref{obstruction-4dim}.
\begin{theorem*}
Let $(M,J)$ be a $4$-dimensional almost complex manifold. Assume that there exists an almost complex blow-up of $(M,J)$ at $x\in M$. Then, the Nijenhuis tensor of $J$ vanishes at $x$.
\end{theorem*}
Holomorphic maps between $4$-dimensional almost complex manifolds have been extensively studied by Draghici, Li and Zhang in \cite{DLZ2} and by Zhang in \cite{Zhang}.

As already mentioned, the complex blow-up at a point of a K\"ahler manifold is still K\"ahler. Thus, it is natural to ask whether the almost K\"ahler condition is also stable under almost complex blow-ups. We prove this is true in dimension $4$ if the almost complex structure of the base manifold satisfies the line condition, see Theorem \ref{almost-kaehler-4-dim}.
\begin{theorem*}
Let $(M,J,g,\omega)$ be a $4$-dimensional compact almost K\"ahler manifold. Assume that $J$ satisfies the line condition at $x\in M$. Then the almost complex blow-up $(\tilde{M}_x,\tilde{J})$ of $M$ at $x$ admits an almost K\"ahler metric $\tilde{\omega}$.
\end{theorem*}

\vskip.1truecm

The paper is organized as follows: in Section \ref{sec-blowup-4}, we start by considering the $4$-dimensional case, by introducing the notion of almost complex structures satisfying the {\em weak line condition at $x$}, respectively {\em line condition at $x$} and of {\em almost complex blow-up} at a point $x$ of a $4$-dimensional almost complex manifold $(M,J)$ (see Definition \ref{almost-complex-blowup-4-dim-x}). Then we prove the existence results for blow-ups (see Theorem \ref{theorem-blowup} and Theorem \ref{theorem-blowup-strong}). We also prove that the vanishing of the Nijenhuis tensor at a point is a necessary condition for the existence of almost complex blow-ups (see Theorem \ref{obstruction-4dim}).

 
Section \ref{Examples} provides examples of explicit almost complex structures on $\C^2$ satisfying the weak line condition, \eqref{extension-condition-1} and \eqref{extension-condition-2} respectively the line condition at a point and an example of almost complex structure satisfying the weak line condition, that does not extend to the blow-up. 

Section $4$ deals with the existence of  almost K\"ahler structures on blow-ups of almost K\"ahler manifolds in dimension $4$ (Theorem \ref{almost-kaehler-4-dim}). The proof proceeds by constructing a closed positive form which is approximately of type $(1,1)$ and then follows arguments of Taubes in \cite{T} to remove the anti-invariant part.

Finally, in Section \ref{sec-blowup-2n} we extend the previous results to the case of a $2n$-dimensional almost complex manifold $(M,J)$ (Theorems \ref{theorem-blowup-2n}, \ref{theorem-blowup-strong-2n} and \ref{uniqueness-blowup-2n}). We note that the computations are notationally heavier comparing to the $4$-dimensional case.
\vspace{0.2cm}\noindent

{\em \underline{Acknowledgments}} The authors would like to thank Weiyi Zhang for useful discussions and comments and for pointing out Lemma \ref{topological-extension}.
\section{Almost complex blow-ups of $4$-dimensional almost complex manifolds at a point}\label{sec-blowup-4}
Let $V$ be a $2n$ dimensional vector space and $J$ be a complex structure on $V$.
\begin{definition}
Let $W$ be a $J$-invariant subspace of $V$. Then $J$ induces a complex structure on the quotient $V\slash W$ by $J[v]:=[Jv]$. 
\end{definition}
We recall the following definition
\begin{definition} Let $(M,J)$, $(M',J')$ be two almost complex manifolds. A smooth map $f:M\to M'$ is said to {\em pseudoholomorphic} if $J'\circ df=df\circ J$. A pseudolomorphic diffeomorphism $f:M\to M'$  is called {\em pseudobiholomorphism}; in such a case $(M,J)$ and $(M',J')$ are said to be {\em pseudobiholomorphic}
\end{definition}
Let $(M,J)$ be a $4$-dimensional almost complex manifold and $x\in M$. We give the following
\begin{definition}\label{line-condition} Let $(M,J)$ be $4$-dimensional almost complex manifold. The almost complex structure $J$ satisfies the {\em weak line condition at} $x\in M$ if there exists a smooth local chart $(\mathcal{U},\phi)$, $0\in\mathcal{U}\subset\C^2$, $\phi:\mathcal{U}\to \phi(\mathcal{U})\subset M$, $\phi(0)=x$, such that for all $(p,q)\in \C^2$ the function 
$f:(\C,i)\to (\C^2,\mathbb{J})$ defined by
$$
f(\zeta)=(p\zeta,q\zeta)
$$
is pseudoholomorphic where defined, with $\mathbb{J}=d\phi^{-1}\circ J\circ d\phi$.
\end{definition}
\begin{definition}\label{strong-line-condition}
Let $(M,J)$ be $4$-dimensional almost complex manifold. The almost complex structure $J$ satisfies the {\em line condition at} $x\in M$ if there exists a smooth local chart $(\mathcal{U},\phi)$, $0\in\mathcal{U}\subset\C^2$, $\phi:\mathcal{U}\to \phi(\mathcal{U})\subset M$, $\phi(0)=x$, such that for all $(p,q)\in \C^2$ the almost complex structure coincides with $i$ on the complex subspace $Span_\C\langle p,q\rangle$ of $T_{(p,q)}\C^2$ spanned by $(p,q)$ and also coincides with $i$ on the quotient $T_{(p,q)}\C^2\slash Span_\C\langle p,q\rangle$.
\end{definition}
We can give the following
\begin{definition}\label{almost-complex-blowup-4-dim-x}
Let $(M,J)$ be a $4$-dimensional almost complex manifold and $x\in M$. Let  $(\mathcal{U},\phi)$ be a smooth local chart at $x$, where $0\in\mathcal{U}\subset\C^2$, $\phi:\mathcal{U}\to \phi(\mathcal{U})\subset M$ and $\phi(0)=x$. Let 
$$
\mathcal{B}:=\Big\{\Big( (z_1,z_2),[v_1:v_2]\Big)\in\mathcal{U}\times\mathbb{P}^{1}(\C)\,\,\,\vert\,\,\, z_1v_2-z_2v_1=0\Big\}\subset\C^2\times \mathbb{P}^{1}(\C)
$$
Then we can define the pair $(\tilde{M}_x,\pi)$, where $\tilde{M}_x$ is the smooth manifold 
$$
\tilde{M}_x=\Big(M\setminus\{x\}\cup \mathcal{B}\Big)\slash\sim
$$
for $\phi(z_1,z_2)\sim ((z_1,z_2),[z_1:z_2])\in\mathcal{B},  (z_1,z_2)\neq(0,0)$ and $\pi:\tilde{M}_x\to M$ is defined as 
$$
\pi((z_1,z_2),[v_1:v_2])=\phi(z_1,z_2),
$$ 
for $((z_1,z_2),[v_1:v_2])\in\mathcal{B}$ and 
$\pi(y)=y$, for $y\in M\setminus\{x\}$.
An {\em almost complex blow-up of $M$ at $x$} with respect to $(\mathcal{U},\phi)$ is an almost complex structure $\tilde{J}$ on $\tilde{M}_x$ such that $\pi$ is a $(\tilde{J},J)$-pseudoholomorphic map.
\end{definition}
Note that $\pi:\tilde{M}_x\setminus \pi^{-1}(x)\to M\setminus \{x\}$ is  a pseudobiholomorphism.
\begin{lemma}\label{topological-extension}
Let $(M,J)$ be a $4$-dimensional almost complex manifold and $x\in M$. Then there exists a blow-up of $(M,J)$ at $x$ as in Definition \ref{blow-up-almost-complex-intro} if and only if there exists an almost complex blow-up of $M$ at $x$ with respect to some chart $(\mathcal{U},\phi)$.
\end{lemma}
\begin{proof}
By definition a blow-up with respect to a chart $(\mathcal{U},\phi)$ is an almost complex blow-up. Conversely, given an almost complex blow-up $((\tilde{M}_x,\tilde{J}),\pi)$ we can smoothly identify a neighbourhood $\mathcal{V}$ of $E$ with 
$$
\mathcal{B}=\Big\{\Big( (z_1,z_2),[v_1:v_2]\Big)\in\mathcal{U}\times\mathbb{P}^{1}(\C)\,\,\,\vert\,\,\, z_1v_2-z_2v_1=0\Big\}\subset\C^2\times \mathbb{P}^{1}(\C)
$$
via a diffeomorphism $\xi:\mathcal{V}\to\mathcal{B}$ mapping $E$ to $\{0\}\times\mathbb{P}^(\C)$. Now we define a chart $(\mathcal{U},\phi)$, $\mathcal{U}\subset\C^2$, by 
$$
\phi(z_1,z_2)=\pi((z_1,z_2),[z_1:z_2]),
$$
when $(z_1,z_2)\neq (0,0)$ and $\phi(0,0)=x$. Then $\phi$ is continuos and holomorphic away from $(0,0)$, hence is holomorphic everywhere. We claim that the blow-up with respect to the local chart $(\mathcal{U},\phi)$ exists and is pseudobiholomorphic to $(\tilde{M}_x,\tilde{J})$. Denote the blow-up with respect to the local chart $(\mathcal{U},\phi)$ by 
$$
\tilde{M}_{x\,\phi}=\Big(M\setminus\{x\}\cup \mathcal{B}\Big)\slash\sim .
$$
Then 
$
\psi:\tilde{M}_{x\,\phi}\to \tilde{M}_x
$ defined as
$$
\left\{
\begin{array}{l}
\psi(y)=\pi^{-1}(y)\qquad \hbox{\rm for}\, y\in M\setminus\{x\}\\[5pt]
\psi(b)=\xi^{-1}(b)
\end{array}
\right.
$$
is a pseudobiholomorphism.
\end{proof}

\begin{rem}\label{chart}
Even as a smooth pair $(\tilde{M}_x,\pi)$ depends on the choice of the local chart $(\mathcal{U},\phi)$. For example, let $M=\C^2$, 
$$
\phi_1(z_1,z_2)=(z_1,z_2), \quad \phi_2(z_1,z_2)=(z_1,z_2+\overline{z}_1^2)
$$
and let $(\tilde{M}_1,\pi_1)$ and $(\tilde{M}_2,\pi_2)$ be as defined in Definition \ref{almost-complex-blowup-4-dim-x} using charts $\phi_1$ and $\phi_2$ respectively. Then we claim that there does not exist a diffeomorphism $\psi:\tilde{M}_1\to\tilde{M}_2$ such that $\pi_2\circ \psi=\pi_1$.\newline 
By contradiction: assume such a $\psi$ exists. Consider the map 
$$
u:\C\to\C^2, \quad u(z)=(z,0).
$$
Then we calculate 
$$
\pi_1^{-1}\circ u:\C\setminus\{0\}\to \mathcal{B}\subset\tilde{M}_1,\quad z\mapsto ((z,0),[1:0]).
$$
Therefore, $\pi_1^{-1}\circ u$ has a removable singularity at $0$. On the other hand, 
$$
\pi_2^{-1}\circ u:\C\setminus\{0\}\to\mathcal{B}
\subset\tilde{M}_2,\quad 
z\mapsto ((z,-\overline{z}^2),[1:-\frac{\overline{z}^2}{z}]),
$$
which has a continuos extension across $0$ but not a smooth extension.\newline
Since $\psi$ is a diffeomorphism,
$$
\psi\circ\pi_1^{-1}\circ u:\C\setminus\{0\}\to \tilde{M}_2 
$$
extends smoothly but
$$
\psi\circ\pi_1^{-1}\circ u=\pi_2^{-1}\circ\pi_2\circ\psi\circ\pi_1^{-1}\circ u=\pi_2^{-1}\circ u
$$
which does not extend, giving a contradiction.
\end{rem}
Let $(M,J)$ be a $4$-dimensional almost complex manifold and $x\in M$. Assume that $J$ satisfies the weak line condition given in Definition \ref{line-condition}. We will give sufficient conditions on the almost complex structure $J$ which guarantee the existence of the almost complex blow-up of $M$ at $x$. Let $(\mathcal{U},\phi)$ be a (smooth) local chart at $x$ as in Definition \ref{line-condition}. Set 
$$
\mathbb{J}:=d\phi^{-1}\circ J\circ d\phi.
$$
Then $\mathbb{J}$ is an almost complex structure on $\mathcal{U}$. Denote by $p_1:\mathcal{U}\times \mathbb{P}^{1}(\C) \to \mathcal{U}$ the natural projection 
and et $p_1^{\mathcal{B}}$ be the restriction of $p_1$ to $\mathcal{B}$; then $(p_1^{\mathcal{B}})^{-1}(0)=\mathbb{P}^{1}(\C)$. Note that $p_1^{\mathcal{B}}:\mathcal{B}\setminus (p_1^{\mathcal{B}})^{-1}(0)\to \mathcal{U}\setminus \{0\}$ is a diffeomorphism and defining 
$$
\tilde{\mathbb{J}}=d(p_1^{\mathcal{B}})^{-1}\circ \mathbb{J}\circ dp_1^{\mathcal{B}}
$$
away from $(p_1^{\mathcal{B}})^{-1}(0)$, it turns out that $p_1^{\mathcal{B}}:(\mathcal{B}\setminus (p_1^{\mathcal{B}})^{-1}(0),\tilde{\mathbb{J}})\to (\mathcal{U}\setminus \{0\},\mathbb{J})$ is a pseudobiholomorphism. The almost complex structure $\mathbb{J}$ can be represented in the local chart $\mathcal{U}$ by a matrix $ 2\times 2$
$$
\mathbb{J}=\begin{bmatrix}
A & B\\
C & D
\end{bmatrix}
$$
where each entry $A, B, C, D$ is a pair $(A^{(1)},A^{(2)})$,$(B^{(1)},B^{(2)})$, $(C^{(1)},C^{(2)})$, $(D^{(1)},D^{(2)})$ of smooth complex valued functions on $\mathcal{U}$. Under this identification, the action of $\mathbb{J}$ at the point $z=(z_1,z_2)$ on the tangent vector $u=^t\!\!(u_1,u_2)$ is given by
\begin{equation}\label{multiplication}
\mathbb{J}_{z}[u]=
\begin{bmatrix}
A(z)\cdot u_1+ B(z)\cdot u_2\\
C(z)\cdot u_1+ D(z)\cdot u_2
\end{bmatrix}
:=
\begin{bmatrix}
A^{(1)}(z)u_1+A^{(2)}(z)\bar{u}_1 +B^{(1)}(z)u_2+B^{(2)}(z)\bar{u}_2  \\
C^{(1)}(z)u_1+C^{(2)}(z)\bar{u}_1 +D^{(1)}(z)u_2+D^{(2)}(z)\bar{u}_2  
\end{bmatrix}.
\end{equation}
Let $\Big((z_1,z_2),[v_1:v_2]\Big)$ be a point in $\mathcal{B}$. On the local chart $v_1\neq 0$ respectively $v_2\neq 0$, setting $z=z_1$ and $w=\frac{v_2}{v_1}$ respectively $z=z_2$ and $w=\frac{v_1}{v_2}$, local coordinates on $\mathcal{B}$ are provided by $(z,zw)$ respectively $(zw,z)$ 
and the local expression of $p_1^{\mathcal{B}}$ is given by
$$
p_1^{\mathcal{B}}((z,w))=(z,zw)
$$ 
respectively 
$$
p_1^{\mathcal{B}}((w,z))=(zw,z)
$$ 
By computing the differential of $p_1^{\mathcal{B}}$ at $(z,w)$, respectively $(w,z)$, we obtain
\begin{equation}\label{differenziale-p}
dp_1^{\mathcal{B}}=
\begin{bmatrix}
1& 0\\
w& z
 \end{bmatrix}
\end{equation}
respectively 
\begin{equation}\label{differenziale-p-1}
dp_1^{\mathcal{B}}=
\begin{bmatrix}
z& w\\
0& 1
 \end{bmatrix}
\end{equation}
The differential of $(p_1^{\mathcal{B}})^{-1}$ at $(z,zw)$ respectively $(zw,z)$ can be directly computed by the above expressions obtaining the following formula
\begin{equation}\label{differenziale-p-inversa}
d(p_1^{\mathcal{B}})^{-1}=\frac{1}{z}
\begin{bmatrix}
z & 0   \\
-w& 1 
 \end{bmatrix}
\end{equation}
and 
\begin{equation}\label{differenziale-p-inversa-1}
d(p_1^{\mathcal{B}})^{-1}=\frac{1}{z}
\begin{bmatrix}
1 & -w   \\
0& z 
 \end{bmatrix}
\end{equation}
Let us focus on the local chart $v_1\neq 0$.
Then, for any given tangent vector $u=^t\!\![a,b]$ formula \eqref{differenziale-p} yields to
\begin{equation}\label{differenziale-tangente}
(dp_1^{\mathcal{B}})_{(z,w)}
\begin{bmatrix}
a\\
b
\end{bmatrix}
=
\begin{bmatrix}
a\\
zb+wa
\end{bmatrix}
\end{equation}
Set
\begin{equation}\label{vector-c-d}
\begin{bmatrix}
c\\
d
\end{bmatrix}:= \mathbb{J}_{(z,zw)}
\begin{bmatrix}
a\\
zb+wa
\end{bmatrix}
\end{equation}
Taking into account \eqref{differenziale-tangente} and \eqref{vector-c-d} we easily compute
\begin{equation}\label{almost-complex-structure-blowup}
\tilde{\mathbb{J}}_{(z,zw)}
\begin{bmatrix}
a\\
b
\end{bmatrix}=
\begin{bmatrix}
c\\
\frac{d-wc}{z}\\
\end{bmatrix}
\end{equation}
\begin{itemize}
\item[i)] Case $a\neq 0,b=0$.\newline
From \eqref{vector-c-d}, we have:
$$
\begin{bmatrix}
c\\
d
\end{bmatrix}:= \mathbb{J}_{(z,zw)}
\begin{bmatrix}
a\\
wa
\end{bmatrix}
=
\begin{bmatrix}
ia\\
iwa
\end{bmatrix}
$$
since $J$ satisfies the weak line condition. Therefore, we deduce that
$$
c=ia,\qquad d=iwa
$$
and consequently
\begin{equation}\label{J-first-case}
\tilde{\mathbb{J}}_{(z,zw)}
\begin{bmatrix}
a\\
0
\end{bmatrix}=
\begin{bmatrix}
ia\\
0
\end{bmatrix}
\end{equation}
\item[ii)] Case $a=0,\,b\neq 0$. From \eqref{vector-c-d}, we have:
$$
\begin{bmatrix}
c\\
d
\end{bmatrix}:= 
\mathbb{J}_{(z,zw)}
\begin{bmatrix}
0\\
zb
\end{bmatrix}
=
\begin{bmatrix}
B\cdot zb\\
D\cdot zb
\end{bmatrix}
$$
Therefore, we obtain the following equations:
$$
\left\{
\begin{array}{ll}
c&= B\cdot zb\\
d&=D\cdot zb
\end{array}
\right.
$$
Consequently, from formula \eqref{almost-complex-structure-blowup}, the explicit expression of the almost complex structure $\tilde{\mathbb{J}}$ is given by
\begin{equation}\label{almost-complex-structure-blowup-explicit}
\tilde{\mathbb{J}}_{(z,w)}
\begin{bmatrix}
0\\
b
\end{bmatrix}=
\begin{bmatrix}
B\cdot zb\\[5pt]
\frac{(D-wB)\cdot zb}{z}
\end{bmatrix}
\end{equation}
Therefore, on the local chart $v_1\neq 0$, taking into account the above expressions and \eqref{J-first-case}, a sufficient condition under which the almost complex strcuture $\tilde{\mathbb{J}}$ can be extended on the whole blow-up is given by 
\begin{equation}\label{extension-condition-1}
D^{(2)}(z,zw)-wB^{(2)}(z,zw)=F(z,w)z
\end{equation}
for a suitable smooth complex valued function $F(z,w)$. \newline
Similarly, on the local chart $v_2\neq 0$, a sufficient condition for the extension of $\tilde{\mathbb{J}}$ to the whole blow-up is given by 
\begin{equation}\label{extension-condition-2}
A^{(2)}(zw,z)-wC^{(2)}(zw,z)=G(w,z)z,
\end{equation}
for a suitable smooth complex valued function $G(z,w)$. 
\end{itemize}
Summing, we have proved the following
\begin{theorem}\label{theorem-blowup}
Let $(M,J)$ be a $4$-dimensional almost complex manifold. If $J$ satisfies the weak line condition at $x\in M$ and conditions \eqref{extension-condition-1} and \eqref{extension-condition-2} hold, then there exists an almost complex blow-up $(\tilde{M}_x,\tilde{J})$ of $(M,J)$ at $x$.
\end{theorem}
\begin{rem}
Assume that $J$ satisfies the weak line condition. Then from fromula \eqref{almost-complex-structure-blowup-explicit}, we see that $\tilde{J}$ extends Lipischitz continuosly across the divisor.
\end{rem}
The following give a sufficient condition for blow-up
\begin{theorem}\label{theorem-blowup-strong}
Let $(M,J)$ be a $4$-dimensional almost complex manifold. If $J$ satisfies the line condition at $x\in M$,  then there exists an almost complex blow-up $(\tilde{M}_x,\tilde{J})$ of $(M,J)$ at $x$.
\end{theorem}
\begin{proof}
We note that $dp_1^B$ is complex linear and maps the complex lines $\{w=\hbox{\rm const}\}$ to lines through $(0,0)\in\C^2$. Therefore, if $J$ satisfies the line condition, then $\tilde{\mathbb{J}}$ coincides with $i$ on both the complex lines $L_c=\{w=c\}$  and on the quotients $\C^2\slash L_c$. This last condition means that 
$$
\tilde{\mathbb{J}}_{(z,w)}
\begin{bmatrix}
0\\
b
\end{bmatrix}=
\begin{bmatrix}
*\\[5pt]
ib
\end{bmatrix}.
$$
Comparing with \eqref{almost-complex-structure-blowup-explicit}, we see that $\tilde{\mathbb{J}}_{(z,w)}
\begin{bmatrix}
0\\
b
\end{bmatrix}$ extends across $z=0$. As shown above see \eqref{J-first-case}, the weak line condition already implies that $\tilde{\mathbb{J}}_{(z,w)}
\begin{bmatrix}
a\\
0
\end{bmatrix}$ extends across $z=0$. Hence $\tilde{\mathbb{J}}$ extends as required.
\end{proof}
The same proof gives the following expression of $\tilde{J}$ in the local charts.
\begin{prop}\label{lemma-3}
Let $(M,J)$ be a $4$-dimensional almost complex manifold. If $J$ satisfies the line condition at $x\in M$,  then in the local chart $v_1\neq 0$,
\begin{equation}\label{eq:tilde_j_strong_cond}
\tilde{\mathbb{J}}_{(z,w)}=
\begin{pmatrix}
(i,0) & (0,B^{(2)})\vspace{0.2cm}\\
(0,0) & (i,0)
\end{pmatrix}
\end{equation}
Moreover, there exists a smooth function $H=H(z,w)$ such that
\begin{equation}
B_2(z,w)=\vert z\vert^2\bar{z}(C+H(z,w)),
\end{equation} 
where $C$ is a real constant.
\begin{proof}
From Theorem \eqref{theorem-blowup-strong}, we have that in the local chart $v_1\neq 0$,
\begin{equation}
\tilde{\mathbb{J}}_{(z,w)}=
\begin{pmatrix}
(i,0) & (B^{(1)},B^{(2)})\vspace{0.2cm}\\
(0,0) & (i,0)
\end{pmatrix}.
\end{equation}
Then,
\begin{equation}
\tilde{\mathbb{J}}_{(z,w)}^2=
\begin{pmatrix}
(-1,0) & (2iB^{(1)},0)\vspace{0.2cm}\\
(0,0) & (-1,0)
\end{pmatrix}.
\end{equation}
Since $\tilde{\mathbb{J}}^{2}=-\id$, then $B^{(1)}=0$, which proves \eqref{eq:tilde_j_strong_cond}.

By definition of almost complex blow-up, we know that $p_1^{\mathcal{B}}\circ (\psi_1)^{-1}$ defines a pseudobiholomorphism between $\left((\psi_1)^{-1}(\mathcal{B}\setminus E),\tilde{\mathbb{J}}\right)$ and $(\mathcal{U}\setminus \{x\},\mathbb{J})$. For the sake of simplicity, we denote $p_1^{\mathcal{B}}\circ (\psi_1)^{-1}:=p_1^{\mathcal{B}}$. The almost complex structure $\mathbb{J}$ at $(z,w)\in\mathcal{U}\setminus\{x\}$ is then given by 
$$\mathbb{J}_{(z,w)}=(dp_1^{\mathcal{B}})_{(z,\frac{w}{z})}\circ \tilde{\mathbb{J}}_{(z,\frac{w}{z})}\circ (dp_1^{\mathcal{B}})_{(z,w)}^{-1}.
$$
Recall that the expressions for $(dp_1^{\mathcal{B}})_{(z,\frac{w}{z})}$ and $(dp_1^{\mathcal{B}})_{(z,w)}^{-1}$ are
\begin{align*}
(dp_1^{\mathcal{B}})_{(z,\frac{w}{z})}=
\begin{pmatrix}
1 & 0\\
\frac{w}{z} & z
\end{pmatrix}, \qquad 
(dp_1^{\mathcal{B}})_{(z,w)}^{-1}=
\begin{pmatrix}
1 & 0\\
-\frac{w}{z^2} & \frac{1}{z}
\end{pmatrix},
\end{align*}
so to obtain, with direct computations,
\begin{align*}
\mathbb{J}_{(z,w)}=\begin{pmatrix}
(i,-B^{(2)}\frac{\bar{w}}{\bar{z}^2}) & (0,B^{(2)}\frac{1}{\bar{z}})\vspace{0.3cm}\\
(0,-B^{(2)}\frac{|w|^2}{z\bar{z}^2}) & (i,B^{(2)} \frac{1}{|z|^2})
\end{pmatrix}
\end{align*}
with $B^{(2)}=B^{(2)}(z,\frac{w}{z})$. Since $\mathbb{J}$ is a smooth tensor and extends to $(0,0)$, each of functions $B^{(2)}\frac{\bar{w}}{\bar{z}^2}$, $B^{(2)}\frac{1}{\bar{z}}$, $B^{(2)}\frac{|w|^2}{z\bar{z}^2}$, and $B^{(2)} \frac{1}{|z|^2}$ have to be smooth. By expanding $B^{(2)}=B^{(2)}(z,\frac{w}{z})$ in Taylor series, one obtains
\begin{align*}
&B^{(2)}=\gamma_0+\gamma_1 z+\gamma_2 \frac{w}{z} +\gamma_{\bar{1}}\bar{z}+\gamma_{\bar{2}}\frac{\bar{w}}{\bar{z}}\\
&+\gamma_{11}z^2+\gamma_{12}w+\gamma_{1\bar{1}}|z|^2+\gamma_{1\bar{2}}z\frac{\bar{w}}{\bar{z}}+\gamma_{22}\frac{w^2}{z^2}+\gamma_{2\bar{1}}\bar{z}\frac{w}{z}+\gamma_{2\bar{2}}\frac{|w|^2}{|z|^2}+\gamma_{\bar{1}\bar{1}}\bar{z}^2+\gamma_{\bar{1}\bar{2}}\bar{w}+\gamma_{\bar{2}\bar{2}}\frac{\bar{w}^2}{\bar{z}^2}\\
&+\gamma_{111}z^3+\gamma_{112}zw+\gamma_{11\bar{1}}|z|^2z+\gamma_{11\bar{2}}z^2\frac{\bar{w}}{\bar{z}}+\gamma_{122}\frac{w^2}{z}+\gamma_{12\bar{1}}\bar{z}w+\gamma_{12\bar{2}}\frac{|w|^2}{\bar{z}}+\gamma_{1\bar{1}\bar{1}}|z|^2\bar{z}\\
&+\gamma_{1\bar{1}\bar{2}}z\bar{w}+\gamma_{1\bar{2}\bar{2}}z\frac{\bar{w}^2}{\bar{z}^2}+\gamma_{222}\frac{w^3}{z^3}+\gamma_{22\bar{1}}\bar{z}\frac{w^2}{z^2}+\gamma_{22\bar{2}}\frac{|w|^2w}{|z^2|z}+\gamma_{2\bar{1}\bar{1}}\bar{z}^2\frac{w}{z}+\gamma_{2\bar{1}\bar{2}}\frac{|w|^2}{z}\\
&+\gamma_{2\bar{2}\bar{2}}\frac{|w|^2\bar{w}}{|z|^2\bar{z}}+\gamma_{\bar{1}\bar{1}\bar{1}}\bar{z}^3+\gamma_{\bar{1}\bar{1}\bar{1}}\bar{z}\bar{w}+\gamma_{\bar{1}\bar{2}\bar{2}}\frac{\bar{w}^2}{\bar{z}}+\gamma_{\bar{2}\bar{2}\bar{2}}\frac{\bar{w}^3}{\bar{z}^3}+ h.o.t.
\end{align*}
Therefore, since $B^{(2)}(z,\frac{w}{z})\cdot\frac{|w|^2}{|z|^2\bar{z}}$ must be a smooth function, the coefficients of the Taylor series of $B^{(2)}(z,\frac{w}{z})$ of order smaller than $2$ must vanish and the only non zero coefficient of order $3$ is $\gamma_{1\bar{1}\bar{1}}$. As a consequence, the function $B^{(2)}(z,w)$ has the desired form, with
\[
B^{(2)}(z,w)=|z|^2\bar{z}(C+H(z,w)),
\]
where $C$ is a constant and $H\in o(1)$.
\end{proof}
\end{prop}
We have the following necessary condition for the extension of $\mathbb{J}$.
\begin{theorem}\label{obstruction-4dim}
Let $(M,J)$ be a $4$-dimensional almost complex manifold. Assume that there exists an almost complex blow-up of $(M,J)$ at $x\in M$. Then, the Nijenhuis tensor of $J$ vanishes at $x$.
\end{theorem}
\begin{proof}
Let $(\mathcal{U},\phi)$ be a local chart at $x$, $\mathcal{U}\subset\C^2$ $\phi:\mathcal{U}\to \phi(\mathcal{U})\subset M$, $\phi(0)=x$. Let $J$ be represented in the local chart as 
\begin{equation}\label{expression-J}
\mathbb{J}_{(z_1,z_2)}=\begin{bmatrix}
(A^{(1)},A^{(2)}) & (B^{(1)},B^{(2)})\\[5pt]
(C^{(1)},C^{(2)})& (D^{(1)},D^{(2)})
\end{bmatrix},
\end{equation}
where $A^{(i)}$, $B^{(j)}$, $C^{(h)}$, $D^{(k)}$ are complex valued smooth functions of $(z_1,z_2)$.

In the local chart $(\C^2,\psi_1)$, the local expression $\tilde{\mathbb{J}}$ of the almost complex structure $\tilde{J}$ on the blow-up is
\[
\tilde{\mathbb{J}}_{(z,w)}=(dp_1^{\mathcal{B}})_{(z,zw)}^{-1}\circ \mathbb{J}_{(z,zw)}\circ (dp_{1}^{\mathcal{B}})_{(z,w)},
\]
so that
\begin{equation}\label{expression-J-tilde}
\tilde{\mathbb{J}}_{(z,w)}\begin{bmatrix}
a \\
b
\end{bmatrix}=\begin{bmatrix}
A^{(1)}a+B^{(1)}(wa+zb)+A^{(2)}\overline{a}+B^{(2)}(\overline{w}\overline{a}+\overline{w}\overline{b})\\
\\
z^{-1}(-A^{(1)}aw-B^{(1)}(wa+zb)w+C^{(1)}a+D^{(1)}(wa+zb)\\
-A^{(2)}\overline{a}+B^{(2)}(\overline{wa}+\overline{z}\overline{b})w+C^{(2)}\overline{a}+D^{(2)}(\overline{wa}+\overline{z}\overline{b})),
\end{bmatrix}
\end{equation}
where $A^{(i)}$, $B^{(j)}$, $C^{(h)}$, $D^{(k)}$ are evaluated at $(z,zw)$.
In particular, when $a=1$, $b=0$
\begin{equation}\label{eq:j_a1_b0}
\tilde{\mathbb{J}}_{(z,w)}\begin{bmatrix}
1 \\
0
\end{bmatrix}=\begin{bmatrix}
A^{(1)}+B^{(1)}w+A^{(2)}+B^{(2)}\overline{w}\\[5pt]
z^{-1}(-A^{(1)}w-B^{(1)}w^2+C^{(1)}+D^{(1)}w -A^{(2)}w-B^{(2)}|w|^2+C^{(2)}+D^{(2)}\overline{w})
\end{bmatrix}
\end{equation}
and when $a=0$, $b=1$,
\begin{equation}\label{eq:j_a0_b1}
\tilde{\mathbb{J}}_{(z,w)}\begin{bmatrix}
0 \\
1
\end{bmatrix}=\begin{bmatrix}
B^{(1)}z+B^{(2)}\overline{z}\\[5pt]
-B^{(1)}w+D_1 +z^{-1}(-B^{(2)}\overline{z}w+D^{(2)}\overline{z})
\end{bmatrix}.
\end{equation}
Under the assumption that $\tilde{\mathbb{J}}$ extends across the exceptional divisor, i.e., $\{z=0\}$, it must hold that
\begin{align}
-A^{(1)}w-B^{(1)}w^2+C^{(1)}+D^{(1)}w -A^{(2)}w-B^{(2)}|w|^2+C^{(2)}+D^{(2)}\overline{w}=zH_1(z,w)\label{eq:ext_cond_1}\\
-B^{(2)}\overline{z}w+D^{(2)}\overline{z}=zH_2(z,w)\label{eq:ext_cond_2}
\end{align}
where $H_1$, $H_2$ are smooth functions.

We may assume that at $(z_1,z_2)=(0,0)$ $\mathbb{J}$ is standard, that is
$$
\mathbb{J}_{(0,0)}=
\begin{bmatrix}
(i,0) & (0,0)\\[5pt]
(0,0) & (i,0)
\end{bmatrix}.
$$

Let us denote by $a_i$, $b_j$, $c_h$, $d_k$ the terms of order higher than $0$ of the Taylor series expansion in $(z_1,z_2)$ of, respectively,  $A^{(i)}=A^{(i)}(z_1,z_2)$, $B^{(j)}=B^{(j)}(z_1,z_2)$, $C^{(h)}=C^{(h)}(z_1,z_2)$, $D^{(k)}=D^{(k)}(z_1,z_2)$, so that we can write
\begin{equation*}
\mathbb{J}_{(z_1,z_2)}=\begin{bmatrix}
(i+a_1,a_2) & (b_1,b_2)\\[5pt]
(c_1,c_2)& (i+d_1,d_2)
\end{bmatrix}.
\end{equation*}
Since $\mathbb{J}^2=-\id$, a straightforward computation shows that the first order terms of $a_1$, $b_1$, $c_1$, $d_1$ vanish.

By examining the Taylor series expansions for the $z^{-1}$-components of \eqref{eq:j_a1_b0} and \eqref{eq:j_a0_b1} and by equations \eqref{eq:ext_cond_1} and \eqref{eq:ext_cond_2}, we derive the following relations at the origin $(0,0)$.

The degree one terms of \eqref{eq:j_a1_b0} and \eqref{eq:ext_cond_2} imply that
\begin{equation*}
(c_2)_{\overline{z_1}}=0.
\end{equation*}
The degree two terms of \eqref{eq:j_a1_b0} and \eqref{eq:j_a0_b1} and extension conditions \eqref{eq:ext_cond_1} and \eqref{eq:ext_cond_2} imply that
\begin{equation*}
(d_2)_{\overline{z_1}}=0, \qquad (a_2)_{\overline{z_1}}=0, \qquad (c_2)_{\overline{z}_2}=0.
\end{equation*}
The degree three terms of \eqref{eq:j_a1_b0} and \eqref{eq:j_a0_b1} and extension conditions \eqref{eq:ext_cond_1} and \eqref{eq:ext_cond_2} imply that
\begin{equation*}
(b_2)_{\overline{z}_1}=0, \qquad (d_2)_{\overline{z}_2}=0, \qquad (a_2)_{\overline{z}_2}=0.
\end{equation*}
The degree four terms of \eqref{eq:j_a1_b0} and extension condition \eqref{eq:ext_cond_1} and imply that
\begin{equation*}
(b_2)_{\overline{z}_2}=0.
\end{equation*}
Hence, at the origin $(0,0)$,
\begin{equation*}
N_{J}\left(\frac{\del}{\del \overline{z}_1},\frac{\del}{\del \overline{z}_2}\right)=0,
\end{equation*} 
which implies that $N_J\restrict{(0,0)}\equiv 0$.

\end{proof}
\begin{rem}\label{non-sufficient}
It has to be remarked that the existence of almost complex blow-up of a $4$-dimensional almost complex manifold $(M,J)$ at a point $x$ depends on the choice of the local chart $(\mathcal{U},\phi)$ at $x$, as the following example shows. \newline
Let $(M,J)$ be a $4$-dimensional almost complex manifold and let $x\in M$. Assume that $J$ is integrable in a neighbourhood of $x$. Without loss of generality we can work in the local model $(\C^2,i)$, $x=0\in\C^2$, with coordinates $(z_1=x_1+iy_1,z_2=x_2+iy_2)$, assuming $J=i$. Let $(\mathcal{U},\phi)$ defined by $\mathcal{U}=\C^2$ and
$$
\phi(z_1,z_2)=(z_1,z_2+\overline{z}_1^2).
$$ 
Then $\phi:\mathcal{U}\to \C^2$ is a global diffeomorphism. Then, on the chart $(\mathcal{U},\phi)$, the local expression of $J=i$ is given by
$$
\mathbb{J}=d\phi^{-1}\circ i\circ d\phi.
$$
Explicitly, $\mathbb{J}_{(z_1,z_2)}$ is given by
$$
\left\{
\begin{array}{l}
\mathbb{J}\partial_{x_1}=\partial_{y_1}+4y_1\partial{x_2}\\[5pt]
\mathbb{J}\partial_{y_1}=-\partial_{x_1}-4y_1\partial{y_2}\\[5pt]
\mathbb{J}\partial_{x_2}=\partial_{y_2}\\[5pt]
\mathbb{J}\partial_{y_2}=-\partial_{x_2}.
\end{array}
\right.
$$
Clearly $\mathbb{J}$ is integrable on $\mathcal{U}$. According to the notation used in formula \eqref{expression-J}, a direct computation shows that 
$$
\mathbb{J}_{(z_1,z_2)}=\begin{bmatrix}
(i,0) & (0,0)\\[3pt]
(0,-2i(z_1-\overline{z}_1)& (i,0)
\end{bmatrix}.
$$
By formula \eqref{expression-J-tilde}, we derive that 
$$
\tilde{\mathbb{J}}_{(z,w)}\begin{bmatrix}
a \\
b
\end{bmatrix}=\begin{bmatrix}
ia\\
\\
ib-2i\frac{(z-\overline{z})}{z}\overline{a}
\end{bmatrix}.
$$
In particular, 
$$
\tilde{\mathbb{J}}_{(z,w)}\begin{bmatrix}
1 \\
0
\end{bmatrix}=\begin{bmatrix}
i\\
\\
-2i\frac{(z-\overline{z})}{z}
\end{bmatrix}.
$$
Therefore, $\tilde{\mathbb{J}}$ cannot be extended continuosly through the exceptional divisor. The expression of $\tilde{\mathbb{J}}$ is given by
$$
\tilde{\mathbb{J}}_{(z,w)}=\begin{bmatrix}
(i,0) & (0,0)\\[5pt]
(0,-2i\frac{(z_1-\overline{z}_1)}{z})& (i,0)
\end{bmatrix}.
$$
The last example shows that the vanishing of the Nijenhuis tensor at a point is not a sufficient condition for the existence of almost complex blow-up in a fixed local chart.
\end{rem}

The next Theorem proves the uniqueness of almost complex blow-ups.
\begin{theorem}\label{uniqueness-blowup-4}
Almost complex blow-ups at a point are unique in the following sense. Suppose $f:(X_1, J_1) \to (X_2, J_2)$ is a pseudobiholomorphism with $f(p_1) = p_2$ and that $X_1$ and $X_2$ have blow ups $\tilde{X}_1$ and $\tilde{X}_2$ at $p_1$ and $p_2$ respectively. Then there exists a pseudobiholomorphism $\tilde{f}:\tilde{X}_1 \to \tilde{X_2}$ such that $p_2 \circ \tilde{f} = f \circ p_2$ and identifying the exceptional divisor with the projectivized tangent spaces we have $\tilde{f}[e] = [d\tilde{f} e]$ for $[e] \in \mathbb{P} T_{p_1}X_1$.
\end{theorem}

\begin{cor} If a blow up of $(X,J)$ exists at the point $p$ then it is independent of the chart used up to pseudobiholomorphism.
\end{cor}
\begin{proof} We apply Theorem \ref{uniqueness-blowup-4} to the identity map.
\end{proof}

By contrast, even in the integrable case, the biholomorphism type of a blow up may depend on the choice of the base point $p$.

\begin{proof}(of Theorem \ref{uniqueness-blowup-4}.) The formulas in the theorem define a bijection between $\tilde{X}_1$ and $\tilde{X_2}$, which is holomorphic away from the exceptional divisor. By elliptic regularity, it suffices to show that $\tilde{f}$ extends continuously across $E$.

We work in local coordinate charts on ${\mathcal B}_1$ and ${\mathcal B}_2$ corresponding to charts on $X_1$ and $X_2$. In terms of these charts we write $f$ as $$f(z_1, z_2) = (f_1(z_1, z_2), f_2(z_1, z_2)).$$ We have $f(0,0)=(0,0)$ and let's assume that on the projectivized tangent space at $(0,0)$ we have $$[df(0,0)(1,a)] = [df_1(0,0)(1,a) : df_2(0,0)(1,a)] = [1:b],$$ that is, we assume the plane $\{z_2 = az_1\}$ is not mapped to $\{z_1=0\}$. Recall that we use local charts which are complex at the origin, thus $df_1(0,0)$ and $df_2(0,0)$ are complex linear maps $\C^2 \to \C$.

We check continuity at the point $( (0,0), [1:a])$.

Near the points $( (0,0), [1:a])$ and $( (0,0), [1:b])$ in ${\mathcal B}_1$ and ${\mathcal B}_2$ we have local coordinates $(z,w)$ centered at $(0,a)$ and $(0,b)$ respectively and corresponding to the points $( (z, zw), [1:w] )$. The exceptional divisor intersects these charts in $\{z=0\}$. When $z \neq 0$ our function $\tilde{f}$ is given in these coordinates by
$$(z,w) \mapsto (f_1(z,zw), \frac{f_2(z, zw)}{f_1(z, zw)}).$$

Expanding in $z$, complex linearity at $(0,0)$ implies
$$f_1(z, zw) = z(\frac{\partial f_1(0,w)}{\partial z_1} + w  \frac{\partial f_1(0,w)}{\partial z_2}) + h.o.t.$$ and similarly for $f_2$.
Letting $z \to 0$, we see that $$(f_1(z,zw), \frac{f_2(z, zw)}{f_1(z, zw)}) \to (0, \frac{\frac{\partial f_2(0,0)}{\partial z_1} + w  \frac{\partial f_2(0,0)}{\partial z_2}}{\frac{\partial f_1(0,0)}{\partial z_1} + w  \frac{\partial f_1(0,0)}{\partial z_2}}).$$

The quotient here is just $\frac{df_2(0,0)(1,w)}{df_1(0,0)(1,w)}$, and since $df_1(0,0)(1,a) \neq 0$ by assumption, this converges to $b$ as $w \to a$ as required.
\end{proof}

\section{Examples}\label{Examples}
In this section we will construct several examples illustrating different situations. 
\begin{enumerate}
\item[I)] {\em $\mathbb{J}$ is smooth on $\C^2$, satisfies the weak line condition and $\tilde{\mathbb{J}}$ extends on the blow-up}.\vskip.2truecm\noindent
Let $\mathbb{J}$ be the smooth $(1,1)$-tensor field on $\C^2$ defined by
$$
\mathbb{J}_{(z,w)}
=\begin{bmatrix}
(i,0) & (0,0)\\[5pt]
(-i\vert z\vert^2\bar{z}w,\,-\sqrt{2+\vert z\vert^4}z\bar{w}) & (i(1+\vert z\vert ^4),\,\vert z\vert^2\sqrt{2+\vert z\vert^4})
\end{bmatrix}.
$$
Then, a straightforward calculation shows that $\mathbb{J}$ gives rise to an almost complex structure on $\C^2$. \newline
\begin{enumerate}
\item[a)] $\mathbb{J}$ satisfies the weak line condition, that is 
$$
\mathbb{J}_{(z,w)}
\begin{bmatrix}
z\\
w
\end{bmatrix}
=
\begin{bmatrix}
iz\\
iw
\end{bmatrix},
$$
for all $(z,w)\in\C^2$. Indeed 
$$
\mathbb{J}_{(z,w)}
\begin{bmatrix}
z\\
w
\end{bmatrix}
=
\begin{bmatrix}
iz\\
-i\vert z\vert^4w-\vert z\vert^2\sqrt{2+\vert z\vert^4}\bar{w}+i(1+\vert z\vert ^4)w+\vert z\vert^2\sqrt{2+\vert z\vert^4}\bar{w}
\end{bmatrix}
=\begin{bmatrix}
iz\\
iw
\end{bmatrix}.
$$
\item[b)] Now we check that conditions \eqref{extension-condition-1} and \eqref{extension-condition-2} are satisfied. We have:
$$
A^{(2)}=0,\quad B^{(2)}=0,\quad C^{(2)}=-\sqrt{2+\vert z\vert^4}z\bar{w}),\quad D^{(2)}=\vert z\vert^2\sqrt{2+\vert z\vert^4}.
$$
Therefore,
$$
D^{(2)}(z,zw)-wB^{(2)}(z,zw)=\vert z\vert^2\sqrt{2+\vert z\vert^4}=F(z,w)z
$$
and 
$$
A^{(2)}(zw,z)-wC^{(2)}(zw,z)=-\vert zw\vert^2\sqrt{2+\vert zw\vert^4}\bar{z}w=G(w,z)z
$$
where we set
$$
F(z,w)=\sqrt{2+\vert zw\vert^4}\bar{z},\qquad 
G(w,z)=-\vert w\vert^2\sqrt{2+\vert zw\vert^4}\bar{z}^2.
$$
\end{enumerate}
Therefore, $\tilde{\mathbb{J}}$ extends on the blow-up.\vskip.3truecm\noindent
\item[II)] {\em $\mathbb{J}$ is smooth on $\C^2$, satifies the weak line condition and $\tilde{\mathbb{J}}$ does not extend on the blow-up}.\vskip.2truecm\noindent
Let $\mathbb{J}$ be the smooth $(1,1)$-tensor field on $\C^2$ defined by
$$
\mathbb{J}_{(z,w)}
=\begin{bmatrix}
(i,0) & (0,0)\\[5pt]
(-i\vert z\vert^6\bar{z}w,\,-\sqrt{2+\vert z\vert^8}\bar{z}^3\bar{w}) & (i(1+\vert z\vert^8),\,\sqrt{2+\vert z\vert^8}\bar{z}^4)
\end{bmatrix}.
$$
Then, a straightforward calculation shows that $\mathbb{J}$ gives rise to an almost complex structure on $\C^2$. \newline 
\begin{enumerate}
\item[a)] As in case I), a direct computation shows that $\mathbb{J}$ satisfies the weak line condition.
\item[b)] The extension formula \eqref{extension-condition-1} is not satisfied.\newline
Indeed, $$
D^{(2)}(z,zw)-wB^{(2)}(z,zw)=\sqrt{2+\vert z\vert^8}\bar{z}^4)
$$
which cannot be written as $F(z,w)z$, with $F$ smooth. 
\end{enumerate}
Summing up, $\mathbb{J}$ is an almost complex structure on $\C^2$, which satisfies the weak line condition and which do not satisfy the extension condition \eqref{extension-condition-1}. Therefore $\tilde{\mathbb{J}}$ does not extend on the blow-up.\vskip.3truecm\noindent
\item[III)] {\em $\mathbb{J}$ is smooth on $\C^2$ and satisfies the line condition.} 
Let $\mathbb{J}$ be the smooth $(1,1)$-tensor field on $\C^2$ defined by
$$
\mathbb{J}_{(z,w)}
=\begin{bmatrix}
(i,-z^2\bar{w}) & (0,\vert z\vert^2z)\\[5pt]
(0,\,-\vert w\vert^2 z) & (i,\,\vert z\vert^2 w)
\end{bmatrix}.
$$
Then one can check directly that $\mathbb{J}^2=-\id$. 
\begin{enumerate}
\item[a)] $\mathbb{J}$ satisfies the line condition. Let $(z,w)$ be any point in $\C^2$. Set 
$$L_{(z,w)}=\Span_\C\langle z,w\rangle,\,\qquad L_{(z,w)}^\perp\Span_\C\langle \bar{w},-\bar{z}\rangle.
$$
Then, we need to show that 
$$
\mathbb{J}_{(z,w)}
\begin{bmatrix}
\bar{w}\\
-\bar{z}
\end{bmatrix}
=
\begin{bmatrix}
i\bar{w}\\
-i\bar{z}
\end{bmatrix}+ H(z,w)
\begin{bmatrix}
z\\
w
\end{bmatrix}.
$$
Indeed, we have:
$$
\mathbb{J}_{(z,w)}
\begin{bmatrix}
\bar{w}\\
-\bar{z}
\end{bmatrix}
=\begin{bmatrix}
i\bar{w}-(\vert w\vert^2+\vert z\vert^2)z^2\\
-i\bar{z}-(\vert w\vert^2+\vert z\vert^2) zw
\end{bmatrix}
=\begin{bmatrix}
i\bar{w}\\
-i\bar{z}
\end{bmatrix}+ H(z,w)
\begin{bmatrix}
z\\
w
\end{bmatrix},
$$
where $H(z,w)=(\vert w\vert^2+\vert z\vert^2) z$. 
\end{enumerate}
Therefore $\mathbb{J}$ satisfies the line condition and according to Theorem \ref{theorem-blowup-strong}, the almost complex structure $\tilde{\mathbb{J}}$ extends to the blow-up.
\newline
In real coordinates $(x,y,u,v)$, where we set $z=x+iy,\,w=u+iv$, the almost complex structure $\mathbb{J}$ can be expressed as
$$
\mathbb{J}\frac{\partial}{\partial x}=\frac{\partial}{\partial y},\quad
\mathbb{J}\frac{\partial}{\partial y}=-\frac{\partial}{\partial x},\quad
\mathbb{J}\frac{\partial}{\partial u}=(x^2+y^2)^2\frac{\partial}{\partial x}+\frac{\partial}{\partial v},\quad
\mathbb{J}\frac{\partial}{\partial v}=-(x^2+y^2)^2\frac{\partial}{\partial y}-\frac{\partial}{\partial u}.
$$
An easy computation shows that the Nijenhuis tensor of $\mathbb{J}$ evaluated at $\frac{\partial}{\partial x},\frac{\partial}{\partial u}$ is given by
$$
N_{\mathbb{J}}(\frac{\partial}{\partial x},\frac{\partial}{\partial u})=4y(x^2+y^2)\frac{\partial}{\partial x}-4y(x^2+y^2)\frac{\partial}{\partial y},
$$
proving that $\mathbb{J}$ is not integrable. Note that, accordingly to Theorem \ref{obstruction-4dim}, 
$N_{\mathbb{J}}\vert_0=0$.
\end{enumerate}
\section{Almost K\"ahler blow-ups of $4$-dimensional almost K\"ahler manifolds}
\begin{theorem}\label{almost-kaehler-4-dim}
Let $(M,J,g,\omega)$ be a $4$-dimensional compact almost K\"ahler manifold. Assume that $J$ satisfies the line condition at $x\in M$. Then the almost complex blow-up $(\tilde{M}_x,\tilde{J})$ of $M$ at $x$ admits an almost K\"ahler metric $\tilde{\omega}$.
\end{theorem}
Before starting with the proof of the theorem, we will need to recall some basic facts on Hodge theory in four dimension. Let $(M,J,g,\omega)$ be a compact $4$-dimensional compact almost Hermitian (non necessarily almost K\"ahler) manifold. Then, the almost complex structure $J$ acts on $2$-forms as an involution via the following 
\[
J\alpha(X,Y):=\alpha(JX,JY),
\]
for every couple of vector fields $X,Y$ on $M$. Denote by \[
\mathcal{A}_J^{\pm}:=\{\alpha\in\mathcal{A}^{2}(M): J\alpha=\pm \alpha\}
\]
and by $\bigwedge_J^{\pm}$ the corresponding bundles.
Let $\mathcal{A}_g^{+}$ and $\mathcal{A}_g^{-}$ be the spaces of self-dual, respectively anti-self-dual, $2$-forms on $M$ and set, as usual, $\bigwedge_g^{+}$, respectively $\bigwedge_g^{-}$, their corresponding bundles. We have the following bundle decomposition
\begin{equation*}
\textstyle\bigwedge_g^+=\R\langle\omega\rangle\oplus \bigwedge_J^-.
\end{equation*} 
For any given $\beta\in\mathcal{A}^2$, we denote by $(\beta)_g^+$, respectively $(\beta)_g^+$, the self-dual, respectively anti-self-dual, component of $\beta$. 
Let $d^*$ be the co-differential. Then we have the following Hodge decomposition, see, e.g., \cite[Lemma 2.4]{DLZ}. 
\begin{lemma}\label{lemma-DLZ}
Let $\alpha\in\mathcal{A}_g^+$. Let $\alpha=\alpha_h+d\theta+d^*\varphi$ be the Hodge decomposition of $\alpha$. Then,
\begin{equation}
(d\theta)_g^+=(d^*\varphi)_g^+, \qquad (d\theta)_{g}^-=-(d^*\varphi)_g^-.
\end{equation}
In particular, the $2$-form $$\alpha-2(d\theta)_g^+=\alpha_h$$ is harmonic and the $2$-form $$\alpha+2(d\theta)_g^-=\alpha_h+2d\theta$$ is closed.
\end{lemma}
Let $\mathcal{H}_g^{\pm}$ be the spaces of harmonic self-dual, respectively, anti-self-dual forms. Denote by $\{\omega_1,\dots,\omega_s\}$ a basis of $\mathcal{H}_{g}^+$. For every $k\in\{1,\dots, s\}$, the harmonic form $\omega_k$ decomposes as $\omega_k=h_{k}\omega+\omega_{k}^{-}$, where $h_k\in\mathcal{C}^{\infty}(M)$ and $\omega_k^-$ is the projection of $\omega_k$ onto $\bigwedge_{J}^-$.

The following lemma ensures that we are always allowed to consider a basis of $\mathcal{H}_{g}^{+}(\tilde{M})$ such that the $L^2$-product
\begin{equation}\label{eq:ort_cond_omega_s-}
\langle \omega_l^-,\omega_m^-\rangle =0, \qquad  l\neq m.
\end{equation}
\begin{lemma}\label{lem:ort_cond_omega_s}
Given a basis $\{\omega_1,\dots, \omega_s\}$ of $\mathcal{H}_{g}^{+}$, the set $\{\omega_{1}',\dots,\omega_{s}'\}$ inductively defined by
$$
\omega_1'=\omega_1,\quad 
\omega_k'=\omega_k-\sum_{n=1, \, \omega_n^{'-}\neq 0}^{k-1}\frac{\langle\omega_{k}^-,\omega_n^{'-}\rangle}{\langle\omega_{n}^{'-},\omega_n^{'-}\rangle}\omega_{n}^{'}, \quad k\in\{1,\dots, s\}.
$$
is a basis of $\mathcal{H}_{\tilde{g}}^{+}$ which satisfies \eqref{eq:ort_cond_omega_s-}.
\end{lemma}
\begin{proof}
Clearly, $\{\omega_k'\}_{k=1}^s$ is a basis of $\mathcal{H}_g^{+}$. We check that \eqref{eq:ort_cond_omega_s-} holds.

Suppose that $1\leq m < l \leq s$ and, by induction, that $\langle\omega_m^{'-},\omega_n^{'-}\rangle=0$, for all $n\neq m$, $n,m\leq l-1$. Then
\begin{align*}
\langle \omega_l^{'-},\omega_m^{'-}\rangle &=\langle(\omega_l-\sum_{n=1, \, \omega_n^{'-}\neq 0}^{l-1}\frac{\langle\omega_{l}^-,\omega_n^{'-}\rangle}{\langle\omega_{n}^{'-},\omega_n^{'-}\rangle}\omega_{n}^{'}), \omega_m^{'-}\rangle\\
&=\langle \omega_l, \omega_m^{'-}\rangle - \sum_{n=1, \, \omega_n^{'-}\neq 0}^{l-1}\frac{\langle\omega_{l}^-,\omega_n^{'-}\rangle}{\langle\omega_{n}^{'-},\omega_n^{'-}\rangle}\langle\omega_{n}^{'}, \omega_m^{'-}\rangle\\
&=\langle \omega_l, \omega_m^{'-}\rangle-\langle\omega_{l}^-,\omega_m^{'-}\rangle=0,
\end{align*}
since the product of $J$-invariant and $J$-anti-invariant forms of degree 2 is zero by bedegree reasons, and hence $\langle \omega_l, \omega_m^{'-}\rangle=\langle \omega_l^-, \omega_m^{'-}\rangle$.
\end{proof}

Before proceeding with the proof Theorem \ref{almost-kaehler-4-dim}, we start by observing that the projection map $\pi:\tilde{M}_x\to M$ is a pseudobiholomorphism away from the exceptional divisor, that is $\pi:(\tilde{M}_x\setminus \pi^{-1}(x),\tilde{J})\to (M\setminus \{x\},J)$ is a pseudobiholomorphism. Accordingly, $\pi^*\omega$ gives rise to a closed positive $(1,1)$-form on $(\tilde{M}_x\setminus \pi^{-1}(x),\tilde{J})$ and $\pi^*\omega$ degenerates on $ \pi^{-1}(x)$. Therefore we need to correct $\pi^*\omega$ by adding a closed positive $(1,1)$-form supported on a neighbourhood of $ \pi^{-1}(x)$. To this purpose we will adapt an argument by Taubes in \cite{T}. On 
$$\mathcal{B}:=\Big\{\Big( (z_1,z_2),[v_1:v_2]\Big)\in\mathcal{U}\times\mathbb{P}^{1}(\C)\,\,\,\vert\,\,\, z_1v_2-z_2v_1=0\Big\}
$$
we have the two local charts $\psi_1:\C^2\to \mathcal{B}$, $\psi_2:\C^2\to\mathcal{B}$ defined as 
$$
\psi_1(z,w)=\Big((z,zw),[1:w]\Big),\quad
\psi_2(w,z)=\Big((zw,z),[w:1]\Big).
$$
Notice that  
$$\psi_2^{-1}\circ\psi_1\vert_{\psi_1^{-1}(\psi_1(\C^2)\cap\psi_2(\C^2))}(\zeta_1,\zeta_2)=\Big(\frac{1}{\zeta_2},\zeta_1\zeta_2\Big).$$
Hence $\mathcal{B}$ has a natural complex structure given by the previous local holomorphic charts $(\psi_j,\C^2)$, $j=1,2$, whose local holomorphic coordinates are denoted by $z=x+iy$, $w=u+iv$. Let us denote by $d'$ and $d''$ the $(1,0)$ respectively $(0,1)$-part of the exterior derivative $d$ with respect to such a complex structure on $\mathcal{B}$.

Let $f\colon (0,+\infty)\rightarrow \R$ any smooth function. Let $\mu\colon \C^2\rightarrow [0,+\infty)$ be given by $\mu(\zeta_1,\zeta_2)=|\zeta_1|^2+|\zeta_2|^2$. Then define
$
F\colon \mathcal{B}\setminus E\rightarrow \R
$
by $$F:=(p_1^{\mathcal{B}})^*\mu^*f.$$

Computing $id' d'' F$ in the chart $(\psi_1,\C^2)$, we obtain the following 
\begin{lemma}\label{lemma-1}
The form 
$
id' d'' \psi_1^*f
$ is closed and it is given by the following formula when $z\neq 0$:
\begin{eqnarray}\label{eq:d'd''psi_f}
id' d'' \psi_1^*F &=& i\Big\{\Big[f''\vert z\vert^2(1+\vert w\vert^2)^2 + f'(1+\vert w\vert^2)\Big]dz\wedge d\bar{z} +\\
&+& \big[f'' \vert z\vert^2 (1+\vert w\vert^2)\bar{z}w+ f'\bar{z}w\Big] dz\wedge d\bar{w}+\nonumber \\
&+& \big[f'' \vert z\vert^2 (1+\vert w\vert^2)z\bar{w}+ f'z\bar{w}\Big] dw\wedge d\bar{z}+\nonumber \\
&+& \Big[ f''\vert z\vert^4 \vert w\vert^2 + f'\vert z\vert^2\Big]dw\wedge d\bar{w}\Big\},\nonumber
\end{eqnarray}
where $f''$ and $f'$ are evaluated at $\vert z\vert^2(1+\vert w\vert^2)$.
\end{lemma}
\begin{proof}
Let $(z,w)$ be local coordinates on $\C^2$, so that the local expression of $\psi_1^*F$ is $\psi_1^*F(z,w)=f\left(|z|^2(1+|w|^2)\right)$. Let us set $t:=|z|^2(1+|w|^2)$. Then, we compute first
\begin{align}\label{d'psi_f}
id''\psi_1^*F&=if'\cdot (\frac{\partial t}{\del \bar{z}}d\bar{z}+\frac{\partial t}{\del \bar{w}}d\bar{w})\\
&=i f'\cdot\left((1+|w|^2)zd\overline{z}+|z|^2wd\overline{w}\right).\nonumber
\end{align}
By applying $d'$ to \eqref{d'psi_f}, one easily obtains \eqref{eq:d'd''psi_f}.
\end{proof}
\begin{rem}\label{rem:fubini_study}
Note that when $f=\log (\vert\zeta_1\vert^2+\vert\zeta_2\vert^2)$, then $id' d'' F$ extends to $\mathcal{B}$ and, once restricted to $\mathbb{P}^1(\C)$, it coincides with  $\omega_{FS}$ on  in the affine coordinate $w=\frac{v_2}{v_1}$.
\end{rem}
Now we want to compute the $\tilde{J}$-anti-invariant component of $id' d'' \psi_1^*f$. We have the following
\begin{lemma}\label{lemma-2}
Set $B_2=P+iQ$. Then the $\tilde{J}$-anti-invariant component of $id' d'' \psi_1^*F$ is expressed by
\begin{eqnarray}\label{eq:J_ant_inv_d'd''psi_f}
i\Big(\psi_1^*d' d'' F- \tilde{J}\psi_1^* d' d'' F  \Big) &=& 2\Big[ f'' \vert z\vert^2(1+\vert w\vert^2)^2 + f' (1+\vert w\vert^2)\Big]\Big[ P(dx\wedge du-dy\wedge dv)\\
&+& Q(dx\wedge dv + dy\wedge du) +(P^2+Q^2)du\wedge dv\Big],\nonumber
\end{eqnarray}
where $f''$ and $f'$ are evaluated at $\vert z\vert^2(1+\vert w\vert^2)$.
\end{lemma}
\begin{proof}
We start by computing $\tilde{J}d'd''\psi_1^*F$. Let $\tilde{\mathbb{J}}$ be the local expression in $(\psi_1)^{-1}(\mathcal{B})\subset\C^2$ of  $\tilde{J}$ on $\tilde{M}_x$. Since $J$ satisfies the line condition at $x\in M$, by Proposition \ref{lemma-3} the almost complex structure $\tilde{\mathbb{J}}$ at $(z,w)$ is represented by the matrix
\begin{equation*}
\tilde{\mathbb{J}}_{(z,w)}=
\begin{pmatrix}
(i,0) & (0,B_2)\vspace{0.2cm}\\
(0,0) & (i,0)
\end{pmatrix}
\end{equation*}
where $B_2=B_2(z,w)$. By passing to the real underlying structure, setting $B_2=P+iQ$, $z=x+iy$, $w=u+iv$, and identifying the vector fields $\del_z=\del_x+i\del_y$, $\del_w=\del_u+i\del_y$, we obtain that the almost complex structure $\tilde{J}$ at $(x,y,u,v)$ is given by
\begin{equation}
\tilde{\mathbb{J}}_{(x,y,u,v)}=
\begin{pmatrix}
0 & -1 & P & Q\\
1 & 0 & Q & -P\\
0 & 0 & 0 & -1\\
0 & 0 & 1 & 0
\end{pmatrix}
\end{equation}
and, hence, the induced almost complex structure on forms (still denoted by $\tilde{\mathbb{J}}$) is represented by
\begin{equation}
\tilde{\mathbb{J}}_{(x,y,u,v)}=
\begin{pmatrix}
0 & 1 & 0 & 0\\
-1 & 0 & 0 & 0\\
P & Q & 0 & 1\\
Q & -P & -1 & 0
\end{pmatrix}.
\end{equation}
Since $\tilde{J}$ is $\mathcal{C}^{\infty}$-linear, we can compute the action of $\tilde{\mathbb{J}}$ directly on forms, taking into account the expression of $id'd''\psi_1^*f$ computed in Lemma \ref{lemma-1}. It is immediate to check that the forms
\begin{align*}
dz\wedge d\bar{w}&=(dx\wedge du+dy\wedge dv),\\
dw\wedge d\bar{z}&=-(dx\wedge du+dy\wedge dv)+i(-dx\wedge dv +dy\wedge du),\\
dw\wedge d\bar{w}&=-2i du\wedge dv
\end{align*}
are $\tilde{\mathbb{J}}$-invariant, i.e., $\tilde{\mathbb{J}}(dz\wedge d\bar{w})=dz\wedge d\bar{w}$ and so on, whereas the form
\[
dz\wedge d\bar{z}=-2i dx\wedge dy
\]
is such that
\[
\tilde{\mathbb{J}}(dz \wedge d\bar{z})=-2i(dx\wedge dy +P(-dx\wedge du+ dy\wedge dv)+Q(-dx\wedge dv- dy\wedge du)-(P^2+Q^2)du\wedge dv).
\]
As a direct consequence, the expression for the $\tilde{J}$-anti-invariant component $i(d'd''\psi_1^*F-\tilde{J}d'd''\psi_1^*F)$ of $id'd''\psi_1^*F$ is exactly \eqref{eq:J_ant_inv_d'd''psi_f}.
\end{proof}
We want to bound the $\tilde{J}$-anti-invariant component of $id' d'' F$. The following two lemma will be useful
\begin{lemma}\label{lemma-4}
Let $0<\varepsilon <1$ and $0<\delta << 1$. Let $\eta=\delta e^{\frac{1}{\varepsilon}}$ Then for $\delta\leq x\leq \eta$, the function 
\begin{equation}\label{function-g}
g(x)=\frac{\varepsilon\log x}{x}-\frac{1+\varepsilon\log\delta}{x}
\end{equation}
satisfies 
\begin{equation}\label{eq:diff_eq_estimates}
g'(x)x+g(x)=\frac{\varepsilon}{x}, 
\end{equation}
and
\begin{equation}\label{eq:diff_eq_boundary_cond}
g(\delta)=-\frac{1}{\delta}, \quad g(\eta)=0.
\end{equation}
\end{lemma}
\begin{proof}
By immediate computations, it turns out for the function $g(x)$ as defined in \eqref{function-g}, satisfies
\[
g'(x)=\frac{\varepsilon-\varepsilon\log{x}}{x^2}+\frac{1+\varepsilon\log{\delta}}{x^2},
\]
so that equation \eqref{eq:diff_eq_estimates} holds for such $g$. Moreover, it is easy to check that the boundary conditions \eqref{eq:diff_eq_boundary_cond} are satisfied.
\end{proof}

\begin{lemma}\label{lemma-5} Let
$\gamma$ and $\nu$ be bump functions with the following properties.

$\gamma(x) =1$ if $x \le \delta$, $\gamma(x) =0$ if $x \ge 2\delta$, $|\gamma'| < \frac{2}{\delta}$, $|\gamma''| < \frac{4}{\delta^2}$;

$\nu(x) =1$ if $x \le \frac{\eta}{2}$, $\nu(x) =0$ if $x \ge \eta$, $|\nu'| < \frac{4}{\eta}$, $|\nu''| < \frac{16}{\eta^2}$.

Then we define
$$\tilde{g}(x) = \nu(x)( g(x) + \gamma(x)(\frac{1}{x} - g)).$$

We have the following.

If $x \le 2\delta$ then $|x\tilde{g}'(x) + \tilde{g}(x)| < \frac{C}{x}$ and $|(x\tilde{g}'(x) + \tilde{g}(x))'| < \frac{C}{x^2}$;

if  $2\delta \le x \le \frac{\eta}{2}$ then $|xg'(x) + g(x)| = \frac{\epsilon}{x}$ and $|(xg'(x) + g(x))'| = \frac{\epsilon}{x^2}$;

if $\frac{\eta}{2}\le x \le \eta$ then $|x\tilde{g}'(x) + \tilde{g}(x)| < \frac{C}{x}$ and $|(x\tilde{g}'(x) + \tilde{g}(x))'| < \frac{C}{x^2}$.

\end{lemma}

\begin{proof} The second statement follows from \ref{eq:diff_eq_estimates} so we carry out the estimates in the other two cases.
We will use the explicit formula
\begin{equation}\label{eq:form_xg}
xg=\epsilon\log(\frac{x}{\delta})-1.
\end{equation}

First assume $x \le 2\delta$. Then we have

\begin{align*}
&  |x\tilde{g}'(x) + \tilde{g}(x)| = |x(g + \gamma( \frac{1}{x} - g))' + g + \gamma( \frac{1}{x} -g) |      \\
& = | xg' + \gamma' - \frac{\gamma}{x} - x \gamma' g - x \gamma g' + g + \frac{\gamma}{x} - \gamma g| \\
& = | \gamma'(1 - xg) - (1 - \gamma)(xg' + g)| =   | \gamma'(1 - xg) - (1 - \gamma)\frac{\epsilon}{x} | \\
&< \frac{2}{\delta} |2-\epsilon\log(\frac{x}{\delta})|+ \frac{\epsilon}{x} < \frac{8+\epsilon}{x}+2\frac{\epsilon}{\delta}\log(\frac{x}{\delta}) < \frac{8+\epsilon}{x}+ 4\frac{\epsilon\log(2)}{x},
\end{align*}
where the first inequality follows from the triangle inequality and \eqref{eq:form_xg} and \eqref{eq:diff_eq_estimates}. 

For the derivative,

\begin{align*}
&  (|x\tilde{g}'(x) + \tilde{g}(x))'| =   | (\gamma'(1 - xg) - (1 - \gamma)\frac{\epsilon}{x})' | \\
& = | \gamma''(1 - xg) - 2\gamma'\frac{\epsilon}{x} + \gamma \frac{\epsilon}{x^2}| \\
& < \frac{4}{\delta^2} |2-\epsilon\log(\frac{x}{\delta})|+ \frac{4\epsilon}{x\delta}+\frac{\epsilon}{x^2} < \frac{32+\epsilon\log(2)+9\epsilon}{x^2}.
\end{align*}

In the second case we assume  $\frac{\eta}{2}\le x \le \eta$. Then

\begin{align*}
&  |x\tilde{g}'(x) + \tilde{g}(x)| = |x(\nu g)' + \nu g |      \\
& = | x \nu g'  + x \nu' g + \nu g  | = |\nu \frac{\epsilon}{x} + x \nu' g| \\
& < \frac{\epsilon}{x} + \frac{4}{\eta}|\epsilon\log(\frac{x}{\delta})-1| < \frac{\epsilon}{x} + \frac{4}{x}|\epsilon\log(\frac{x}{\delta})-1|\\
&< \frac{\epsilon(1+\log(2))}{x}. 
\end{align*}

And for the derivative,

\begin{align*}
&  (|x\tilde{g}'(x) + \tilde{g}(x))'| =   |(   \nu \frac{\epsilon}{x} + x \nu' g   )' |      \\
& = |  \nu' \frac{\epsilon}{x} - \nu \frac{\epsilon}{x^2} + \nu' g + x \nu'' g + x \nu' g'| \\
&=| - \nu \frac{\epsilon}{x^2} +  \nu' \frac{2\epsilon}{x}  + x \nu'' g| \\
& <  \frac{\epsilon}{x^2} + \frac{8 \epsilon}{x \eta}  +  \frac{16\epsilon\log(2)}{\eta^2} <  \frac{9\epsilon+ 16\epsilon\log(2)}{x^2}.
\end{align*}

\end{proof}

\begin{proof}[Proof of Theorem \ref{almost-kaehler-4-dim}]
Let us set
\begin{equation}
f(x):=\log(\delta)+\int_{\delta}^x\tilde{g}(y)dy,
\end{equation}
which defines a smooth function $f\colon (0,+\infty)\rightarrow \R$ such that $f(x)=\log(x)$, when $x\leq \delta$ and $f'(x)=0$, for $\eta\leq x$. Define $\Omega:=\pi^*\omega+id'd'' F$, where $F=(p_1^{\mathcal{B}})^*\mu^* f$ as above. Note that, whereas $F$ is defined on $\mathcal{B}\setminus E$, by Remark \ref{rem:fubini_study}, the form $id'd'' F$ extends to $\mathcal{B}$ as a closed $2$-form, which is non degenerate on $E$. Since $\pi^*\omega$ is positive on any $1$-dimensional subspace $\Span_{\C}\langle v \rangle$ with $v\notin TE$, then $\Omega$ is positive on $\tilde{M}$.

Let $\Omega^-$ denote the $\tilde{J}$-anti-invariant component of $\Omega$, i.e.,  the ortogonal projection of $\Omega$ onto $\bigwedge_{\tilde{J}}^{-}$. Let $\tilde{h}$ be a $\tilde{J}$-Hermitian metric on $\tilde{M}$ and let $s:=b^+=\dim H_{dR}^{2}(\tilde{M};\R)$. All norms will be defined with respect to this metric on $\tilde{M}$. We estimate the norms in the chart $(\psi_1,\C^2)$ when $|w|<2$, but it is easy to see that the same bounds apply for the chart $(\psi_2,\C^2)$ when $|w|<2$ and $\mathcal{B}$ is covered by these two charts. 

Using the explicit expression of $\Omega^-$ from Lemma \ref{lemma-2}, since $$(1+|w|^2)(f''|z|^2(1+|w|^2)+f')=(1+|w|^2)(x\tilde{g}'(x)+\tilde{g}(x)),$$ at $x=|z|^2(1+|w|^2)$, the estimates from Lemma \ref{lemma-4} and Proposition \ref{lemma-3}, implies that
\begin{equation}\label{eq:omega_minus_bound}
|\Omega^-| < \frac{C}{|z|^2}|z|^3=C|z|
\end{equation}
and
\begin{align}\label{eq:nabla_omega_minus_bound}
|\nabla\Omega^-| &< |(x\tilde{g}(x)+\tilde{g}(x))'|\cdot|z|\cdot|B^{(2)}|+|(x\tilde{g}(x)+\tilde{g}(x)|\cdot |\nabla B^{(2)}|\\
& < \frac{C}{|z|^4}|z|\cdot|z|^3+ \frac{C}{|z|^2}|z|^2=C.\nonumber
\end{align}

Let us then suppose that $\{\omega_1,\dots,\omega_s\}$ is a basis as in Lemma \ref{lem:ort_cond_omega_s}. Define the quantities
\begin{gather}
\chi_{k}:=\langle\Omega^-,\omega_k\rangle=\int_{\tilde{M}}\Omega\wedge \omega_{k},\\
z_{k,J}:=\langle\omega_{k}^-,\omega_{k}\rangle=\int_{\tilde{M}}\omega_{k}^-\wedge\omega_k
\end{gather}
for $k\in\{1,\dots,s\}$. Applying estimates \eqref{eq:omega_minus_bound}, we have that
\begin{equation}\label{eq:chi_estimates}
|\chi_{k}|\leq C\cdot\varepsilon, \qquad  k\in\{1,\dots,s\},
\end{equation}
since if $\Omega^-\neq 0$, then $|z|< \epsilon$.

Let us then consider the form $$\Omega^--\sum_{k=1,z_{k,J}\neq 0}^s\chi_{k}\,z_{k,J}^{-1}\,\omega_k^-.$$ We claim that this form belongs to $(\mathcal{H}_{\tilde{h}}^+)^{\perp}$.
We begin by observing that if $z_{k,J}=0$ for every $k\in\{1,\dots,s\}$, then $\omega_{k}^-=0$, i.e., each $\omega_k=h_k\tilde{\omega}$ is of type $(1,1)$. Then
$$
\langle\Omega^-,\omega_{m}\rangle=0, \qquad m\in\{1,\dots, s\}
$$
since $\Omega^-$ is of type $(2,0)+(0,2)$.

Let us then suppose that $z_{k,J}\neq 0$ for some indices  and arrange them so that $\{1,\dots,r\}$ are the indices such that $z_{k,J}\neq 0$, for $k\in\{1,\dots, r\}$. We then compute, for $m\in \{1,\dots, s\}$,
\begin{align*}
\left\langle \Omega^--\sum_{k=1}^r\chi_{k}z_{k,J}^{-1}\omega_{k}^-,\omega_{m} \right\rangle&=\langle\Omega^-,\omega_{m}\rangle-\sum_{k=1}^r\chi_{k}z_{k,J}^{-1}\langle\omega_{k}^-,\omega_m\rangle\\
&=\chi_{m}-\chi_{m}z_{m,J}^{-1}\langle\omega_{m}^-,\omega_m\rangle\\
&=0,
\end{align*}
by Lemma \ref{lem:ort_cond_omega_s}. Therefore, $\Omega^--\sum_{k=1}^r\chi_{k}z_{k,J}^{-1}\omega_{k}^-\in(\mathcal{H}_{\tilde{h}}^+)^{\perp}$, as claimed.

Next, by \cite[Lemma 2.4]{DLZ}, see Lemma \ref{lemma-DLZ}, there exists $a\in\mathcal{A}^1(\tilde{M})$ such that
\[
\Omega^--\sum_{k=1}^r\chi_{k}z_{k,J}^{-1}\omega_{k}^-=(da)_{\tilde{h}}^+.
\]
By \eqref{eq:omega_minus_bound}, \eqref{eq:nabla_omega_minus_bound}, and \eqref{eq:chi_estimates}, then \cite[Lemma 1.1]{T} implies that we can choose $a$ such that $|\nabla a|< C \epsilon $.
Set $$\Psi:=\Omega-\sum_{k=1}^r\chi_{k}z_{k,J}^{-1}\omega_{k}+da.$$
Then, by construction, $\Psi$ is closed, the $\tilde{J}$-anti-invariant component of $\Psi$ is zero, so that $\Psi$ provides a closed $(1,1)$-form on $\tilde{M}$. Since $|\nabla a|< C \epsilon$, $\Psi$ is positive for small enough $\epsilon$. Hence, $(\tilde{J},\Psi)$ is an almost K\"ahler structure on $\tilde{M}_x$ and $(\tilde{M}_x,\tilde{J},\Psi)$ is an almost K\"ahler manifold.
\end{proof}
\section{Almost complex blow-ups  of almost complex manifolds at a point}\label{sec-blowup-2n}
In Section \ref{sec-blowup-4}, we defined the blow-up at a point of a $4$-dimensional almost complex manifold. 
In this Section we will extend such a construction for an arbitrary $2n$-dimensional almost complex manifold $(M,J)$. The computations are notationally heavier. We start by giving the following two definitions analogous to Definitions \ref{line-condition} and 
\ref{strong-line-condition}.
\begin{definition}\label{line-condition-2n} Let $(M,J)$ be $2n$-dimensional almost complex manifold. The almost complex structure $J$ satisfies the {\em weak line condition at} $x\in M$ if there exists a smooth local chart $(\mathcal{U},\phi)$, $0\in\mathcal{U}\subset\C^n$, $\phi:\mathcal{U}\to \phi(\mathcal{U})\subset M$, $\phi(0)=x$, such that for all $(q_1,\ldots, q_n)\in \C^n$ the function 
$f:(\C,i)\to (\C^n,\mathbb{J})$ defined by
$$
f(z)=(q_1z,\ldots,q_nz)
$$
is pseudoholomorphic where defined, with $\mathbb{J}=d\phi^{-1}\circ J\circ d\phi$.
\end{definition}
\begin{definition}\label{strong-line-condition-2n}
Let $(M,J)$ be $2n$-dimensional almost complex manifold. The almost complex structure $J$ satisfies the {\em line condition at} $x\in M$ if there exists a smooth local chart $(\mathcal{U},\phi)$, $0\in\mathcal{U}\subset\C^n$, $\phi:\mathcal{U}\to \phi(\mathcal{U})\subset M$, $\phi(0)=x$, such that for all $(p,q)\in \C^n$ the almost complex structure coincides with $i$ on the complex subspace $Span_\C\langle p,q\rangle$ of $T_{(p,q)}\C^n$ spanned by $(p,q)$ and also coincides with $i$ on the quotient $T_{(p,q)}\C^n\slash Span_\C\langle p,q\rangle$.
\end{definition}
We give the following
\begin{definition}
Let $(M,J)$ be an almost complex manifold and $x\in M$. Let  $(\mathcal{U},\phi)$ be a smooth local chart at $x$, where $0\in\mathcal{U}\subset\C^n$, $\phi:\mathcal{U}\to \phi(\mathcal{U})\subset M$ and $\phi(0)=x$. 
Let 
$$
\Big\{\Big( (z_1,\ldots ,z_n),[v_1:\ldots ,v_n]\Big)\in\mathcal{U}\times\mathbb{P}^{n-1}(\C)\,\,\,\vert\,\,\, z_iv_j-z_jv_i=0\,\,\quad i,j=1,\ldots ,n\Big\}\subset\mathcal{U}\times\mathbb{P}^{n-1}(\C)
$$
Then we can define the pair $(\tilde{M}_x,\pi)$, where $\tilde{M}_x$ is the smooth manifold 
$$
\tilde{M}_x=\Big(M\setminus\{x\}\cup \mathcal{B}\Big)\slash\sim
$$
for $\phi(z_1,\ldots z_n)\sim ((z_1,\ldots ,z_n),[z_1:\ldots :z_n])\in\mathcal{B},  (z_1,\ldots,z_n)\neq(0,\ldots,0)$ and $\pi:\tilde{M}_x\to M$ is defined as 
$$
\pi((z_1,\ldots,z_n),[v_1:\ldots :v_n])=\phi(z_1,\ldots ,z_n),
$$ 
for $((z_1,\ldots ,z_n),[v_1:\ldots :v_n])\in\mathcal{B}$ and 
$\pi(y)=y$, for $y\in M\setminus\{x\}$.
An {\em almost complex blow-up of $M$ at $x$} with respect to $(\mathcal{U},\phi)$ is an almost complex structure $\tilde{J}$ on $\tilde{M}_x$ such that $\pi$ is a $(\tilde{J},J)$-pseudoholomorphic map.
\end{definition}
Note that $\pi:\tilde{M}\setminus \pi^{-1}(x)\to M\setminus \{x\}$ is  a pseudobiholomorphism.\newline 
The analogue of Lemma \ref{topological-extension} also holds in higher dimension.\newline
Let $(M,J)$ be an almost complex manifold and $x\in M$. Assume that $J$ satisfies the weak line condition given in Definition \ref{line-condition-2n}. We will give sufficient conditions on the almost complex structure $J$ which guarantee the existence of the almost complex blow-up of $M$ at $p$. Let $(\mathcal{U},\phi)$ be a (smooth) local chart at $x$ as in Definition \ref{line-condition-2n}. Set 
$$
\mathbb{J}:=d\phi^{-1}\circ J\circ d\phi.
$$
Then $\mathbb{J}$ is an almost complex structure on $\mathcal{U}$. Denote by $p_1:\mathcal{U}\times \mathbb{P}^{n-1}(\C) \to \mathcal{U}$ the natural projection and let 
$$
\mathcal{B}:=\Big\{\Big( (z_1,\ldots ,z_n),[v_1:\ldots ,v_n]\Big)\in\mathcal{U}\times\mathbb{P}^{n-1}(\C)\,\,\,\vert\,\,\, z_iv_j-z_jv_i=0\,\,\quad i,j=1,\ldots ,n\Big\}.
$$
Let $p_1^{\mathcal{B}}$ be the restriction of $p_1$ to $\mathcal{B}$; then $(p_1^{\mathcal{B}})^{-1}(0)=\mathbb{P}^{n-1}(\C)$. Note that $p_1^{\mathcal{B}}:\mathcal{B}\setminus (p_1^{\mathcal{B}})^{-1}(0)\to \mathcal{U}\setminus \{0\}$ is a diffeomorphism and defining 
$$
\tilde{\mathbb{J}}=d(p_1^{\mathcal{B}})^{-1}\circ \mathbb{J}\circ dp_1^{\mathcal{B}}
$$
away from $(p_1^{\mathcal{B}})^{-1}(0)$, it turns out that $p_1^{\mathcal{B}}:(\mathcal{B}\setminus (p_1^{\mathcal{B}})^{-1}(0),\tilde{\mathbb{J}})\to (\mathcal{U}\setminus \{0\},\mathbb{J})$ is a pseudobiholomorphism. The almost complex structure $\mathbb{J}$ can be represented in the local chart $\mathcal{U}$ by a matrix $ n\times n$ $A=(A_{ij})$, where each $A_{ij}$ is a pair $(A_{ij}^{(1)},A_{ij}^{(2)})$ of smooth complex valued functions on $\mathcal{U}$. Under this identification, the action of $\mathbb{J}$ at the point $z=(z_1,\ldots,z_n)$ on the tangent vector $u=^t\!\!(u_1,\ldots,u_n)$ is given by
\begin{equation}\label{multiplication-2n}
\mathbb{J}(z)[u]=
\begin{bmatrix}
\sum_{j=1}^nA_{1j}(z)\cdot u_j\\
\vdots\\
\sum_{j=1}^nA_{nj}(z)\cdot u_j
\end{bmatrix}
:=
\begin{bmatrix}
\sum_{j=1}^nA^{(1)}_{1j}(z)u_j+\sum_{j=1}^nA^{(2)}_{1j}(z)\bar{u}_j\\
\vdots\\
\sum_{j=1}^nA^{(1)}_{nj}(z)u_j+\sum_{j=1}^nA^{(2)}_{nj}(z)\bar{u}_j\\
\end{bmatrix}.
\end{equation}
Let $\Big((z_1,\ldots,z_n),[v_1:\ldots :v_n]\Big)$ be a point in $\mathcal{B}$. For $v_j\neq 0$, setting $w_i=\frac{v_i}{v_j}$, local coordinates on $\mathcal{B}$ are provided by 
$$
(w_1,\ldots, w_{j-1},z_j,w_{j+1},\ldots, w_n)
$$
and the local expression of $p_1^{\mathcal{B}}$ is given by
$$
p_1^{\mathcal{B}}((w_1,\ldots, w_{j-1},z_j,w_{j+1},\ldots, w_n))=(z_jw_1,\ldots, z_jw_{j-1},z_j,z_jw_{j+1},\ldots, z_jw_n).
$$ 
By computing the differential of $p_1^{\mathcal{B}}$ at $(w_1,\ldots, w_{j-1},z_j,w_{j+1},\ldots, w_n)$, we obtain
\begin{equation}\label{differenziale-p-2n}
dp_1^{\mathcal{B}}=
\begin{bmatrix}
z_j & 0 & \cdots & \cdots & 0 & w_1 & 0 &\cdots & \cdots &  0\\
0 & z_j & 0 & \cdots & 0 & w_2 & \vdots & \cdots & \cdots &  \vdots\\
\vdots & \ddots & \ddots &\ddots &  \vdots & \vdots & \vdots & \cdots & \cdots &  \vdots\\
\vdots & \cdots & \ddots & z_j & 0 & \vdots & \vdots & \cdots  & \cdots & \vdots\\
\vdots & \cdots & \cdots & \ddots & z_j & w_{j-1} & \vdots & \cdots & \cdots  &  \vdots\vspace{0.1cm}\\
\vdots & \cdots & \cdots & \cdots & 0 & 1 & 0 & \cdots & \cdots &  \vdots \\
\vdots & \cdots & \cdots & \cdots  & \vdots & w_{j+1} & z_j & 0 & \cdots &  \vdots\\
\vdots  & \cdots & \cdots & \cdots & \vdots & \vdots & 0 & \ddots & \ddots & \vdots \\
\vdots & \cdots & \cdots & \cdots & \vdots & \vdots & \vdots & \ddots & \ddots & 0 \\
0 & \cdots & \cdots & \cdots & 0 &  w_{n} & 0 & \cdots &  0 & z_j
 \end{bmatrix}
\end{equation}
The differential of $(p_1^{\mathcal{B}})^{-1}$ at $(z_jw_1,\ldots, z_jw_{j-1},z_j,z_jw_{j+1},\ldots, z_jw_n)$ can be directly computed by the above expression obtaining the following formula
\begin{equation}\label{differenziale-p-inversa-2n}
d(p_1^{\mathcal{B}})^{-1}=\frac{1}{z_j}
\begin{bmatrix}
1 & 0 & \cdots & \cdots & 0 & -w_1 & 0 &\cdots & \cdots &  0\\
0 & 1 & 0 & \cdots & 0 & -w_2 & \vdots & \cdots & \cdots &  \vdots\\
\vdots & \ddots & \ddots &\ddots &  \vdots & \vdots & \vdots & \cdots & \cdots &  \vdots\\
\vdots & \cdots & \ddots & 1 & 0 & \vdots & \vdots & \cdots  & \cdots & \vdots\\
\vdots & \cdots & \cdots & \ddots & 1 & -w_{j-1} & \vdots & \cdots & \cdots  &  \vdots\vspace{0.1cm}\\
\vdots & \cdots & \cdots & \cdots & 0 & z_j & 0 & \cdots & \cdots &  \vdots \\
\vdots & \cdots & \cdots & \cdots  & \vdots & -w_{j+1} & 1 & 0 & \cdots &  \vdots\\
\vdots  & \cdots & \cdots & \cdots & \vdots & \vdots & 0 & \ddots & \ddots & \vdots \\
\vdots & \cdots & \cdots & \cdots & \vdots & \vdots & \vdots & \ddots & \ddots & 0 \\
0 & \cdots & \cdots & \cdots & 0 &  -w_{n} & 0 & \cdots &  0 & 1
 \end{bmatrix}
\end{equation}
Then, for any given tangent vector $u=^t\!\![b_1,\ldots, b_{j-1},a,b_{j+1},\ldots, b_n]$ formula \eqref{differenziale-p-2n} yields to
\begin{equation}\label{differenziale-tangente-2n}
dp_1^{\mathcal{B}}
\begin{pmatrix}
w_1\\
\vdots\\
w_{j-1}\\
z_j\\
w_{j+1}\\
\vdots\\
w_n
\end{pmatrix}
\begin{bmatrix}
b_1\\
\vdots\\
b_{j-1}\\
a\\
b_{j+1}\\
\vdots\\
b_n
\end{bmatrix}
=
\begin{bmatrix}
z_jb_1+w_1a\\
\vdots\\
z_jb_{j-1}+w_{j-1}a\\
a\\
z_jb_{j+1}+w_{j+1}a\\
\vdots\\
z_jb_n+w_na
\end{bmatrix}
\end{equation}
Set
\begin{equation}\label{vector-c-d-2n}
\begin{bmatrix}
d_1\\
\vdots\\
d_{j-1}\\
c\\
d_{j+1}\\
\vdots\\
d_n
\end{bmatrix}:= \mathbb{J}\begin{pmatrix}
z_jw_1\\
\vdots\\
z_jw_{j-1}\\
z_j\\
z_jw_{j+1}\\
\vdots\\
z_jw_n
\end{pmatrix}
\begin{bmatrix}
z_jb_1+w_1a\\
\vdots\\
z_jb_{j-1}+w_{j-1}a\\
a\\
z_jb_{j+1}+w_{j+1}a\\
\vdots\\
z_jb_n+w_na
\end{bmatrix}
\end{equation}
Taking into account \eqref{differenziale-tangente-2n} and \eqref{vector-c-d-2n} we easily compute
\begin{equation}\label{almost-complex-structure-blowup-2n}
\tilde{\mathbb{J}}\begin{pmatrix}
w_1\\
\vdots\\
w_{j-1}\\
z_j\\
w_{j+1}\\
\vdots\\
w_n
\end{pmatrix}
\begin{bmatrix}
b_1\\
\vdots\\
b_{j-1}\\
a\\
b_{j+1}\\
\vdots\\
b_n
\end{bmatrix}=
\begin{bmatrix}
\frac{d_1-w_1c}{z_j}\\
\vdots\\
\frac{d_{j-1}-w_{j-1}c}{z_j}\\
c\\
\frac{d_{j+1}-w_{j+1}c}{z_j}\\
\vdots\\
\frac{d_{n}-w_{n}c}{z_j}
\end{bmatrix}
\end{equation}
\begin{itemize}
\item[i)] Case $a\neq 0,b_k=0,k\in\{1,\ldots ,\hat{j},\ldots ,n\}$.\newline
From \eqref{vector-c-d-2n}, we have:
$$
\begin{bmatrix}
d_1\\
\vdots\\
d_{j-1}\\
c\\
d_{j+1}\\
\vdots\\
d_n
\end{bmatrix}:= \mathbb{J}\begin{pmatrix}
z_jw_1\\
\vdots\\
z_jw_{j-1}\\
z_j\\
z_jw_{j+1}\\
\vdots\\
z_jw_n
\end{pmatrix}
\begin{bmatrix}
w_1a\\
\vdots\\
w_{j-1}a\\
a\\
w_{j+1}a\\
\vdots\\
w_na
\end{bmatrix}
=
\begin{bmatrix}
iw_1a\\
\vdots\\
iw_{j-1}a\\
ia\\
iw_{j+1}a\\
\vdots\\
iw_na
\end{bmatrix}
$$
since $J$ satisfies the line condition.
\item[ii)] Case $a=0,\,b_k\neq 0$ for a given $k\in \{1,\ldots ,\hat{j},\ldots ,n\}$ and $ b_h=0$, for all $h\in\{1,\ldots ,\hat{j},\ldots ,n\}, h\neq k$. From \eqref{vector-c-d-2n}, we have:
$$
\begin{bmatrix}
d_1\\
\vdots\\
d_{j-1}\\
c\\
d_{j+1}\\
\vdots\\
d_n
\end{bmatrix}:= 
\mathbb{J}\begin{pmatrix}
z_jw_1\\
\vdots\\
z_jw_{j-1}\\
z_j\\
z_jw_{j+1}\\
\vdots\\
z_jw_n
\end{pmatrix}
\begin{bmatrix}
0\\
\vdots\\
z_jb_k\\
\vdots\\
0\\
\vdots\\
0
\end{bmatrix}
=
\begin{bmatrix}
A_{1k}\cdot z_jb_k\\
\vdots\\
\vdots\\
\vdots\\
\vdots\\
\vdots\\
A_{nk}\cdot z_jb_k
\end{bmatrix}
$$
Therefore, we obtain the following equations:
$$
\left\{
\begin{array}{ll}
d_1&= A_{1k}\cdot z_j b_k\\
\vdots\\
d_{j-1}&= A_{j-1\,k}\cdot z_j b_k\\
c&=A_{jk}\cdot z_jb_k\\
d_{j+1}&= A_{j+1\,k}\cdot z_j b_k\\
\vdots\\
d_{n}&= A_{n\,k}\cdot z_j b_k
\end{array}
\right.
$$
Consequently, from formula \eqref{almost-complex-structure-blowup-2n}, the explicit expression of the almost complex structure $\tilde{\mathbb{J}}$ is given by
\begin{equation}\label{almost-complex-structure-blowup-explicit-2n}
\tilde{\mathbb{J}}\begin{pmatrix}
w_1\\
\vdots\\
w_{j-1}\\
z_j\\
w_{j+1}\\
\vdots\\
w_n
\end{pmatrix}
\begin{bmatrix}
0\\
\vdots\\
z_jb_k\\
0\\
0\\
\vdots\\
0
\end{bmatrix}=
\begin{bmatrix}
\frac{(A_{1k}-w_1A_{jk})\cdot z_jb_k}{z_j}\\
\vdots\\
\frac{(A_{j-1\,k}-w_{j-1}A_{jk})\cdot z_jb_k}{z_j}\\
A_{jk}\cdot z_jb_k\\
\frac{(A_{j+1\,k}-w_{j+1}A_{jk})\cdot z_jb_k}{z_j}\\
\vdots\\
\frac{(A_{nk}-w_{n}A_{jk})\cdot z_jb_k}{z_j}
\end{bmatrix}
\end{equation}
Therefore, on the local chart $v_j\neq 0$, taking into account the above expressions, a sufficient condition in order that the almost complex strcuture $\tilde{\mathbb{J}}$ can be extended on the whole blow-up is given by 
\begin{equation}\label{extension-condition-2n}
A_{ij}^{(2)}-w_jA_{jk}^{(2)}=F_i^{k,j}(z,w)z_j
\end{equation}
where $j\in\{1,\ldots,n\},\,\,i,k\in\{1,\ldots,\hat{j}\ldots,n\},\,\,$, the functions $A_{ik}^{(2)}$ and $A_{jk}^{(2)}$ are evaluated at the point $(w_1z_j,\ldots , w_{j-1}z_j,z_j,w_{j+1}z_j,\ldots , w_nz_j)$ and $(z,w)=(w_1,\ldots , w_{j-1},z_j,w_{j+1},\ldots , w_n)$ and the functions $F_i^{k,j}(z,w)$ are smooth.
\end{itemize}
Summing, we have proved the following
\begin{theorem}\label{theorem-blowup-2n}
Let $(M,J)$ be an almost complex manifold. If $J$ satisfies the weak line condition at $x\in M$ and conditions \eqref{extension-condition-2n} hold, then there exists an almost complex blow-up $(\tilde{M}_x,\tilde{J})$ of $(M,J)$ at $x$.
\end{theorem}
Similar arguments to the $4$-dimensional case (see Theorem \ref{theorem-blowup-strong}, Theorem \ref{uniqueness-blowup-4}) prove the following
\begin{theorem}\label{theorem-blowup-strong-2n}
Let $(M,J)$ be a $2n$-dimensional almost complex manifold. If $J$ satisfies the line condition at $x\in M$,  then there exists an almost complex blow-up $(\tilde{M}_x,\tilde{J})$ of $(M,J)$ at $x$.
\end{theorem}
\begin{theorem}\label{uniqueness-blowup-2n}
Almost complex blow-ups at a point are unique.
\end{theorem}

\end{document}